\documentclass[12pt]{article}

\usepackage{tikz}
\usepackage{empheq}
\usepackage[shortlabels]{enumitem}
\setlist[enumerate]{leftmargin=15mm,nosep}

\usepackage[doc]{optional}
\usepackage{xcolor}
\usepackage[colorlinks=true,
            linkcolor=refkey,
            urlcolor=lblue,
            citecolor=red]{hyperref}
\usepackage{float}
\usepackage{soul}
\usepackage{graphicx}
\definecolor{labelkey}{rgb}{0,0.08,0.45}
\definecolor{refkey}{rgb}{0,0.6,0.0}
\definecolor{Brown}{rgb}{0.45,0.0,0.05}
\definecolor{lime}{rgb}{0.00,0.8,0.0}
\definecolor{lblue}{rgb}{0.5,0.5,0.99}

 \usepackage{mathpazo}

\colorlet{hlcyan}{cyan!30}

\usepackage{stmaryrd}
\usepackage{amssymb}
\makeatletter
\def\namedlabel#1#2{\begingroup
   \def\@currentlabel{#2}%
   \label{#1}\endgroup
}
\makeatother

\newcommand{\seppthree}{\setlength{\itemsep}{-3pt}}

\newcommand{\J}[1]{\ensuremath{{\operatorname{J}}_{#1}}}
\newcommand{\R}[1]{\ensuremath{{\operatorname{R}}_%
{#1}}}
\newcommand{\Pj}[1]{\ensuremath{{\operatorname{P}}_%
{#1}}}
\newcommand{\Px}[1]{\ensuremath{{\operatorname{P}}_%
{#1}}}
\newcommand{\Nc}[1]{\ensuremath{{\operatorname{N}}_%
{#1}}}

\newcommand{\dist}[1]{\ensuremath{{\operatorname{dist}}_%
{#1}}}

\newcommand{\Pxg}{\ensuremath{\Px{g}}}

\newcommand{\bX}{\ensuremath{\mathbf{X}}}
\newcommand{\bx}{\ensuremath{\mathbf{x}}}
\newcommand{\bZ}{\ensuremath{\mathbf{Z}}}
\newcommand{\bu}{\ensuremath{\mathbf{u}}}
\newcommand{\bj}{\ensuremath{\mathbf{j}}}
\newcommand{\be}{\ensuremath{e}}
\newcommand{\by}{\ensuremath{\mathbf{y}}}
\newcommand{\bzero}{\ensuremath{{\boldsymbol{0}}}}

\providecommand{\siff}{\Leftrightarrow}
\newcommand{\weakly}{\ensuremath{\:{\rightharpoonup}\:}}

\newcommand{\nnn}{\ensuremath{{n\in{\mathbb N}}}}

\newcommand{\menge}[2]{\big\{{#1}~\big |~{#2}\big\}}

\newcommand{\To}{\ensuremath{\rightrightarrows}}

\newcommand{\fenv}[1]%
{\ensuremath{\,\overrightarrow{\operatorname{env}}_{#1}}}
\newcommand{\benv}[1]%
{\ensuremath{\,\overleftarrow{\operatorname{env}}_{#1}}}

\newcommand{\scal}[2]{\left\langle{#1},{#2}  \right\rangle}
\newcommand{\bscal}[2]{\big\langle{#1},{#2}  \big\rangle}

\newcommand{\RR}{\ensuremath{\mathbb R}}
\newcommand{\RP}{\ensuremath{\mathbb{R}_+}}

\newcommand{\RX}{\ensuremath{\,\left]-\infty,+\infty\right]}}

\newcommand{\dom}{\ensuremath{\operatorname{dom}}}

\newcommand{\argmin}{\ensuremath{\operatorname{argmin}}}

\newcommand{\ran}{\ensuremath{{\operatorname{ran}}\,}}
\newcommand{\zer}{\ensuremath{\operatorname{zer}}}

\newcommand{\cdom}{\ensuremath{\overline{\operatorname{dom}}\,}}
\newcommand{\Fix}{\ensuremath{\operatorname{Fix}}}
\newcommand{\Id}{\ensuremath{\operatorname{Id}}}
\newcommand{\bDelta}{{\bf \Delta}}
\newcommand{\bv}{{\bf v}}
\newcommand{\bT}{{\bf T}}
\newcommand{\bD}{{\bf D}}
\newcommand{\bg}{{\bf g}}
\newcommand{\Lss}{{\ensuremath{L}}}
\newcommand{\Icons}{\ensuremath{J}}
\newcommand{\Iobj}{\ensuremath{I \smallsetminus J}}

\newcommand{\pinf}{\ensuremath{+\infty}}

\newcommand{\minimize}[2]{\ensuremath{\underset{\substack{{#1}}}{\mathrm{minimize}}\;\;#2 }}

\newcommand{\vD}{v_D}

\newcommand{\vR}{v_R}

\newcommand{\vI}{v}

\providecommand{\fejer}{Fej\'{e}r}
{\begin{list}{}{%
\settowidth{\labelwidth}{\textrm{#1~}}%
\setlength{\leftmargin}{\labelwidth+\labelsep}}}
{\end{list}}
\usepackage{amsthm}
\usepackage[capitalize,nameinlink]{cleveref}
\crefname{equation}{}{equations}
\crefname{chapter}{Appendix}{chapters}
\crefname{item}{}{items}
\crefname{enumi}{}{}
\newtheorem{theorem}{Theorem}[section]
\newtheorem{lemma}[theorem]{Lemma}

\newtheorem{corollary}[theorem]{Corollary}

\newtheorem{proposition}[theorem]{Proposition}

\newtheorem{example}[theorem]{Example}
\newtheorem{ex}[theorem]{Example}
\newtheorem{fact}[theorem]{Fact}
\newtheorem{remark}[theorem]{Remark}

\providecommand{\abs}[1]{\lvert#1\rvert}
\providecommand{\norm}[1]{\lVert#1\rVert}

\providecommand{\RA}{\Rightarrow}

\providecommand{\grad}{\nabla}

\providecommand{\RR}{\mathbb{R}}

\providecommand{\ran}{\operatorname{ran}}

\providecommand{\dom}{\operatorname{dom}}

\newcommand{\fix}{\ensuremath{\operatorname{Fix}}}

\providecommand{\epi}{\operatorname{epi}}

\providecommand{\Id}{\operatorname{{ Id}}}

\providecommand{\fady}{\varnothing}

\providecommand{\argmin}{\mathrm{arg}\!\min}

\providecommand{\To}{\rightrightarrows}

\providecommand{\fix}{\operatorname{Fix}}
\providecommand{\ran}{\operatorname{ran}}

\providecommand{\Id}{\operatorname{Id}}

\providecommand{\zer}{\operatorname{zer}}
\providecommand{\R}{{ R}}

\providecommand{\fady}{\varnothing}

\newcommand{\cran}{\ensuremath{\overline{\operatorname{ran}}\,}}

\providecommand{\RR}{\mathbb{R}}

\definecolor{myblue}{rgb}{.8, .8, 1}
  \newcommand*\mybluebox[1]{%
    \colorbox{myblue}{\hspace{1em}#1\hspace{1em}}}

\allowdisplaybreaks 

\begin{document}

%

\author{
Heinz H.\ Bauschke\thanks{
Mathematics, University
of British Columbia,
Kelowna, B.C.\ V1V~1V7, Canada. E-mail:
\texttt{heinz.bauschke@ubc.ca}.}~~~and~
Walaa M.\ Moursi\thanks{
Department of Combinatorics and Optimization, University of Waterloo,
Waterloo, Ontario N2L 3G1, Canada.
  E-mail: \texttt{walaa.moursi@uwaterloo.ca}.}
}

\title{\textsf{
On the behaviour of the Douglas--Rachford algorithm\\ for
minimizing a convex function subject\\ to a linear constraint
}
}

\date{July 9, 2020}

\maketitle

\begin{abstract}
The Douglas-Rachford algorithm (DRA) is a powerful optimization method
for minimizing the sum of two convex (not necessarily smooth) 
functions. The vast majority of previous research dealt with the 
case when the sum has at least one minimizer. 
In the absence of minimizers, it was recently shown that
for the case of two indicator functions, 
the DRA converges to a best approximation solution.
In this paper, we present a new convergence result on 
the DRA applied to the problem of minimizing a convex function
subject to a linear constraint. 
Indeed, a normal solution may be found even when the domain of the
objective function and the linear subspace constraint have no point in common.
As an important application, a new parallel splitting result is provided. 
We also illustrate our results through various examples. 
\end{abstract}
{ 


\section{Introduction}

Throughout, we assume that 
\begin{empheq}[box=\mybluebox]{equation}
\text{$X$ is
a 
real Hilbert space, 
}
\end{empheq}
with inner product 
$\scal{\cdot}{\cdot}\colon X\times X\to\RR$ 
and induced norm $\|\cdot\|$. 
We furthermore assume that 
\begin{empheq}[box=\mybluebox]{equation}
\text{$U$ is a closed linear subspace of $X$,}
\end{empheq}
and that 
\begin{empheq}[box=\mybluebox]{equation}
\text{$g\colon X\to\RX$ is convex, lower semicontinuous, and proper.}
\end{empheq}

Our aim is to discuss the behaviour of the
Douglas--Rachford algorithm \cite{DR} applied to solving the 
optimization problem\footnote{
 Let us point out that if $\widetilde{U} = \widetilde{u}+U$
is an \emph{affine} subspace and $\widetilde{g}$ is convex, 
lower semicontinuous, and proper, then 
all our results are applicable by working 
with $U$ and $g = \widetilde{g}(\cdot - \widetilde{u})$ instead. 
}
\begin{equation}
\label{e:origprob}
\minimize{x\in X}{\iota_U(x) + g(x)},
\end{equation}
where $\iota_U(x) = 0$ if $x\in U$ and $\iota_U(x)=+\infty$ if $x\notin U$. 
Note that we do \emph{not} assume a priori that 
\cref{e:origprob} has a solution.
Given any starting point $x_0\in X$, 
the Douglas--Rachford algorithm generates 
the so-called \emph{governing sequence}
\begin{equation}
(T^nx_0)_\nnn
\end{equation}
where 
\begin{empheq}[box=\mybluebox]{equation}
\label{e:defT}
T = \Id-\Pj{U} + \Pxg\R{U}
\end{empheq}
is the Douglas--Rachford operator,
$\Pj{U}$ is the projector of $U$,
$\Pxg$ is the \emph{proximal mapping}
of the function $g$, and
$\R{U} = 2\Pj{U}-\Id = \Pj{U}-\Pj{U^\perp}$
is the reflector of $U$. 
The basic convergence result
(see \cite{LM}, \cite{EckBer}, and \cite{Svaiter}),
guarantees that the \emph{shadow sequence}
\begin{equation}
(\Pj{U}T^nx_0)_\nnn
\end{equation}
converges weakly to a solution of \eqref{e:origprob}
provided that $(\Nc{U}+\partial g)^{-1}(0)\neq\varnothing$.

To deal with the potential lack of solutions
of \eqref{e:origprob}, we define
the \emph{minimal displacement vector}
\begin{empheq}[box=\mybluebox]{equation}
\label{e:defv}
v = \Pj{\cran(\Id-T)}(0).
\end{empheq}
This vector is well defined because
$\cran(\Id-T)$ is convex, closed, and trivially nonempty.
We now assume that the so-called \emph{normal problem} corresponding
to \eqref{e:origprob}, which asks to find a zero
of the operator $-v+\Nc{U}+\partial g(\cdot-v)$, admits at least one 
\emph{normal solution}\footnote{ 
  Note that it is possible that 
$Z$ is empty: 
indeed, consider the case when $X=\RR=U$ and $g=\exp$. 
In this case, $|T^nx|\to+\infty$ for every $x\in\RR$.} 
(see \cite[Definition~3.7]{Sicon}):
\begin{empheq}[box=\mybluebox]{equation}
\label{e:Znonempty}
Z = \menge{x\in X}{v\in \Nc{U}(x)+\partial g(x-v)}\neq\varnothing.
\end{empheq}
We also assume throughout that 
\begin{empheq}[box=\mybluebox]{equation}
\label{e:PZweakly}
\Pj{Z} \text{ is weak-to-weak continuous,}
\end{empheq}
which is automatically the case when $X$ is finite-dimensional, and
that 
\begin{empheq}[box=\mybluebox]{equation}
\label{e:CQ}
0\in U^\perp + \dom g^*,
\end{empheq}
which is a rather mild constraint qualification that is
satisfied, for instance, if $g$ has minimizers\footnote{
 Also note that \eqref{e:CQ} implies that 
the Fenchel dual of \eqref{e:origprob} is feasible and hence
that \eqref{e:origprob} is implicitly assumed to be bounded below.}.
Note that if \eqref{e:origprob} has a solution
and $\partial(\iota_U+g)=\Nc{U}+\partial g$ (this sum formula
is typically guaranteed through a regularity condition), 
then $v=0$ and $Z=\argmin(\iota_U+ g)$. 
Our main result (see \cref{main} below) can now be concisely stated as follows:
{\it Under the above assumptions, which we assume for the rest of the paper,
we have 
\begin{equation}
\label{e:super}
\Pj{U}T^nx_0 \weakly \text{some minimizer of $\iota_U+g(\cdot-v)$.}
\end{equation}
}This is a completely new (and very beautiful) variant of the classical result 
which is proven with a careful function value analysis in \cref{sec:dyn}!
{\it It reveals the Douglas--Rachford algorithm to be a method for solving
the following bilevel optimization problem:}
first, obtain the gap vector between $U=\dom\iota_U$ and $\dom g$.
This level is purely geometrical, depending on the sets $U$ and $\dom g$,
and revealing the minimal displacement vector $v$.
Secondly, if $v\neq 0$, rather than minimizing the original $\iota_U+g $
which would have the optimal value $ +\infty$,
we then instead minimize the minimal perturbation function 
$\iota_U+g(\cdot -v)$.
This has consequences for minimizing the sum
of convex function by using a product space technique;
in fact, real world applications inspired this research 
(see the last section). 

Let us now comment on related previous works which 
will illustrate the complementary nature of the present work. 
To the best of our knowledge, none of these works contains 
the result \eqref{e:super} in the generality of the setting of 
\cref{main}.
The paper \cite{Banjac} by Banjac, Goulart, Stellato, and Boyd 
applies the Douglas--Rachford algorithm with the function $f$ 
being the sum of a quadratic function and the indicator function of 
an affine subspace rather 
than $\iota_U$ and with $g$ being the indicator function of a 
nonempty closed convex set. The Douglas--Rachford method (equivalent to 
ADMM in this setting) is shown to be useful in providing 
certificates of infeasibility. 
The paper \cite{BDM:ORL16} 
concerns the more restrictive case when $g$ is the indicator function of a 
nonempty closed convex set; however, the underlying assumptions there 
do not require \eqref{e:PZweakly}. 
The paper \cite{Sicon} introduces the normal problem 
but it does not contain any algorithmic/dynamic results. 
Similarly to \cite{BDM:ORL16}, the paper 
\cite{101} deals with the case when 
$g$ is assumed to be an indicator function of a closed affine subspace. 
Under suitable assumptions, 
the shadow sequence $(\Pj{U}T^nx_0)_\nnn$ is shown to converge strongly. 
The paper \cite{BM:MPA17} considers an infinite-dimensional setting 
that encompasses two indicator functions; however, 
our present main result is not covered by these results 
(see \cref{r:whatsnew} below). 
In the paper \cite{Liu} by Liu, Ryu, and Yin, the authors study the 
behaviour of the Douglas--Rachford algorithm applied to 
conic programming where $g$ is the indicator function of a nonempty closed 
convex cone while $\iota_U$ is replaced by the sum of a linear function and 
the indicator function of an affine subspace. The Douglas--Rachford method 
is shown to reveal information on the type of pathologies the conic 
program may exhibit.
Finally, the paper \cite{ucla} 
by Ryu, Liu, and Yin  
is the first to provide a comprehensive 
function-value analysis in pathological cases. 
It differs from the present work in that 
Ryu et al.\ allow for a general function $f$ rather than the indicator function 
$\iota_U$ considered here. However, our main result \cref{main} gives 
information on the iterates and the function values that are not covered 
by the results in \cite{ucla} when strong duality fails. 

The remainder of this paper is organized as follows.
In Section~\ref{sec:aux} we review known facts and present 
new auxiliary results that are needed in the main analysis.
Section~\ref{sec:static} presents new descriptions of the 
minimal displacement vector and the set of minimizers 
which are crucial in the convergence proofs. 
The building blocks of our analysis and the main result are presented 
in Sections~\ref{sec:dyn}~and~\ref{sec:main} respectively.
In the final Section~\ref{sec:app}, 
we provide a useful application of
our theory to describe the behaviour of a parallel splitting method.

We employ standard notation from convex analysis and optimization
as can be found, e.g., in \cite{BC} and \cite{Rock70}.

\section{Known and new auxiliary results}
\label{sec:aux}
Because $Z\neq\varnothing$ (see \cref{e:Znonempty}), 
the generalized fixed point set introduced in
\cite{Sicon}
is very well behaved in the sense that 
\begin{empheq}[box=\mybluebox]{equation}
\label{e:defF}
F := \Fix T(\cdot+v) = \menge{x\in X}{x=T(x+v)}
\text{~is convex, closed, and nonempty.}
\end{empheq}
The Douglas--Rachford operator $T$ defined in \eqref{e:defT} 
enjoys the following nice properties which also underline
the importance of $F$ for understanding
the Douglas--Rachford algorithm:

\begin{fact}
\label{Froof}
Let $x\in X$ and $y\in F$. 
Then\footnote{\textcolor{black}{We point out that  \cref{Froof} holds
in the more general setting when $T$ is any firmly nonexpansive mapping.}} 
\begin{equation}
\label{Froof1}
(\forall\nnn)\quad T^ny = y-nv;
\end{equation}
the sequence $(nv+T^nx)_\nnn$ is \emph{\fejer\ monotone} with respect to $F$, i.e., 
\begin{equation}
\label{Froof2}
(\forall\nnn)\quad
\|(n+1)v+T^{n+1}x-y\| \leq \|nv+T^nx-y\|;
\end{equation}
\begin{equation}
\label{Froof4}
\sum_{n=0}^{+\infty}\|T^{n+1}x-T^nx-v\|^2 < +\infty,
\end{equation}
\begin{equation}
\label{e:190517b}
T^nx-T^{n+1}x\to v;
\end{equation}
and the limit 
\begin{equation}
\label{Froof3}
\lim_{n\to+\infty} \Pj{F}(nv+T^nx) \in F
\end{equation}
exists.
\end{fact}
\begin{proof}
See \cite[Corollary~4.2]{BM:MPA17},
\cite[Proposition~2.5(vi)]{101}
and \cite[Proposition~5.7]{BC}.
\end{proof}

Before we proceed, we recall the following useful fact 
that will be used in the proofs of \cref{prop:vD:vR} and \cref{p:vUperp!}. 

\begin{fact}
\label{fact:BB94:Cor4.6}
Let $C$ be a nonempty closed convex subset 
 of $X$.
 Set $w=\Pj{\overline{U-C}}(0)$ and let $x\in X$. Then 
$w = \lim_{n\to\infty}(\Pj{U}-\Id)(\Pj{C}\Pj{U})^nx \in 
\cran(\Pj{U}-\Id)=-U^\perp = U^\perp$.
\end{fact}
\begin{proof}
See  \cite[Corollary~4.6]{BB94}.
\end{proof}

The next result will also be used in the proof of \cref{p:vUperp!}.

\begin{proposition}
\label{prop:vD:vR}
Let $C_1$ and $C_2$ be nonempty closed convex subsets of $X$,
and set $S_1\coloneqq U-C_1$ and $S_2\coloneqq U^\perp-C_2$.
Define
\begin{equation}
\vD\coloneqq \Pj{\overline{S_1}}(0),
\quad
\vR\coloneqq \Pj{\overline{S_2}}(0),
\quad
\vI\coloneqq \Pj{\overline{S_1} \cap \overline{S_2}}(0).
\end{equation}
Then the following hold:
\begin{enumerate}
\item 
\label{prop:vD:vR:0}
$(\vD,\vR)\in U^\perp\times U$.
\item 
\label{prop:vD:vR:i}
$\Pj{U^\perp}(\overline{S_1})\subseteq \overline{S_1}$.
\item 
\label{prop:vD:vR:ii}
$\Pj{U}(\overline{S_2})\subseteq \overline{S_2}$.
\item
\label{prop:vD:vR:iii}
$\vD+\vR\in \overline{S_1} \cap \overline{S_2}$.

\item
\label{prop:vD:vR:v}
$\vI=\vD+\vR$.
\end{enumerate}
\end{proposition}
\begin{proof}
\ref{prop:vD:vR:0}:
Apply \cref{fact:BB94:Cor4.6} with $(C,w)$ replaced 
by $(C_1, \vD)$ (respectively $(C,w)$ replaced 
by $(C_2, \vR)$).
\ref{prop:vD:vR:i}:
Let $y\in \overline{S_1}$.
Then there exist $(u_n)_\nnn$ in $U$
and $(c_{1,n})_\nnn$ is $C_1$ such that
$u_n-c_{1,n}\to y$.
Now, $\Pj{U^\perp}y\leftarrow \Pj{U^\perp}(u_n-c_{1,n})
=-\Pj{U^\perp}c_{1,n}
=\Pj{U}c_{1,n}-c_{1,n}\in U-C_1$.
Hence, $\Pj{U^\perp}y\in \overline{U-C_1}=\overline{S_1}$
and the claim follows.
\ref{prop:vD:vR:ii}:
Proceed similar to the proof of \ref{prop:vD:vR:i}.
\ref{prop:vD:vR:iii}:
Indeed, note that by \ref{prop:vD:vR:0}
we have $\vR\in U$, hence
$\vD+\vR\in \overline{S_1}+\vR
=\overline{U-C_1}+\vR
=\overline{U-C_1+\vR}
=\overline{U-C_1}=\overline{S_1}$.
Similarly, we show that 
$\vD+\vR\in \overline{S_2}$
and the conclusion follows.
\ref{prop:vD:vR:v}:
Note that \ref{prop:vD:vR:i} \& \ref{prop:vD:vR:ii}
imply that $(\Pj{U}\vI, \Pj{U^\perp}\vI)\in 
\overline{S_2}\times \overline{S_1}$.
Consequently,
$\norm{\vR}\le\norm{\Pj{U}\vI}$
 and 
$\norm{\vD}\le\norm{\Pj{U^\perp}\vI}$.
Altogether, in view of 
\ref{prop:vD:vR:0},
 we learn that
 $\norm{\vD+\vR}^2
 =\norm{\vD}^2+\norm{\vR}^2
 \le \norm{\Pj{U}\vI}^2+\norm{\Pj{U^\perp}\vI}^2
 =\norm{\vI}^2$.
 Combining this with 
\ref{prop:vD:vR:iii}, 
 and the definition of $\vI$, we obtain the result. 
\end{proof}

The following simple result, which relies on the assumption
that $U$ is a closed linear subspace, will be used in the proof of 
\cref{main}. 

\begin{lemma}
\label{restproj}
Let $C$ be a nonempty closed convex subset of $U$.
Then
\begin{equation}
\label{e:restproj}
\Pj{C} = \Pj{C}\circ \Pj{U}
\end{equation}
\end{lemma}
\begin{proof}
Let $x\in X$ and let $c\in C\subseteq U$.
Then $\Pj{C}\Pj{U}x\in C$ and
\begin{subequations}
\begin{align}
\scal{c-\Pj{C}\Pj{U}x}{x-\Pj{C}\Pj{U}x}
&=\bscal{\underbrace{c-\Pj{C}\Pj{U}x}_{\in U}}{\underbrace{x-\Pj{U}x}_{\in U^\perp}}
+\scal{c-\Pj{C}\Pj{U}x}{\Pj{U}x-\Pj{C}\Pj{U}x}\\
&= \scal{c-\Pj{C}\Pj{U}x}{\Pj{U}x-\Pj{C}\Pj{U}x}\\
&\leq 0,
\end{align}
\end{subequations}
and we are done.
\end{proof}

We now turn to the minimization of
a convex function subject to a linear constraint.
The following result will be used in the proof of \cref{0601c}.

\begin{lemma}
\label{0601a}
Let $h\colon X\to\RX$ be a proper lower semicontinuous convex function.
Furthermore, let $x$ and $y$ be points in $U$, and let $x^*\in X$.
Then the following hold:
\begin{enumerate}
\item 
\label{0601a1}
If $U^\perp\cap \partial h(x)\neq\varnothing$, then $x$ is a minimizer
of $\iota_U+h$. 
\item 
\label{0601a2}
If $x^*\in U^\perp\cap\partial h(x)$ and $y$ is a minimizer of 
$\iota_U+h$, then $x^*\in U^\perp\cap \partial h(y)$. 
\end{enumerate}
\end{lemma}
\begin{proof}
\cref{0601a1}:
Suppose that $U^\perp\cap \partial h(x)\neq\varnothing$.
Then, since $U^\perp$ is a subspace,
$(-U^\perp)\cap\partial h(x)\neq \varnothing$.
Suppose that $x^*\in \partial h(x)$. 
Then 
$-x^*\in U^\perp = \Nc{U}(x)$. 
It follows that
$0 = (-x^*)+x^*\in \Nc{U}(x)+\partial h(x)
= \partial \iota_U(x)+\partial h(x)
\subseteq \partial(\iota_U+h)(x)$.
By Fermat's rule, $x$ is a minimizer of $\iota_U+h$. 

\cref{0601a2}:
Suppose that $x^*\in U^\perp\cap\partial h(x)\neq\varnothing$. 
Then 
\begin{equation}
\label{e:0601a2b}
(\forall z\in X)\quad
h(z) \geq h(x)+\scal{z-x}{x^*}.
\end{equation}
and 
\begin{equation}
\label{e:0601a2c}
\scal{y-x}{x^*} = 0. 
\end{equation}
On the other hand, 
because $y$ is a minimizer of  $\iota_U+h$, 
we learn from \cref{0601a1} that
\begin{equation}
\label{e:0601a2a}
h(x)=h(y).
\end{equation}
Altogether,
\begin{subequations}
\begin{align}
(\forall z\in X)\quad h(z) 
&\geq h(x)+\scal{z-x}{x^*}\\
&= h(y) + \scal{z-y}{x^*} + \scal{y-x}{x^*}\\
&= h(y) + \scal{z-y}{x^*}.
\end{align}
\end{subequations}
Therefore, $x^*\in\partial h(y)$. 
\end{proof}

The assumption that $U^\perp\cap\partial h(x)\neq\varnothing$
in \cref{0601a}\cref{0601a2} is critical:
\begin{example}
Suppose that $X=\RR$, that $U=\{0\}$,
and that $h(\xi)=-\sqrt{\xi}$, if $\xi\geq 0$ and 
$h(\xi)=\pinf$ if $\xi<0$. 
Then $0$ minimizes $\iota_U+h = \iota_U$ yet
$U^\perp \cap \partial h(0) = \partial h(0)=\varnothing$.
\end{example}

\begin{remark}
\label{rem:q}
Let $h\colon X\to\RX$ be a proper lower semicontinuous convex function.
Then \cref{0601a} implies that the set-valued operator
\begin{equation}
\argmin (\iota_U+h)\To X\colon x\mapsto U^\perp\cap\partial h(x)
\end{equation}
is constant. 
\end{remark}


\section{New static results}
\label{sec:static}
\textcolor{black}{
We start with the following useful 
result for the minimal displacement vector $v$ from \eqref{e:defv}.
\begin{proposition} 
\label{p:vUperp!}
Set $w = \Pj{\overline{U-{\dom g}}}(0)$.
Then the following hold:
\begin{enumerate}
	\item
	\label{p:vUperp!:i}
	$w\in U^\perp$. 
	\item 
	\label{p:vUperp!:ii}
	If $X$ is finite-dimensional, then 
	$v=w = \Pj{\overline{U-\dom g}}(0)\in U^\perp$.
\end{enumerate}
\end{proposition}
\begin{proof}	
	Clearly 
	$\overline{U-\dom g}=\overline{U-\overline{\dom }g}$
	and, 
	$\overline{U^\perp+\dom g^*}=\overline{U^\perp+\overline{\dom} g^*}$.
	\ref{p:vUperp!:i}:
Apply \cref{fact:BB94:Cor4.6}  with $C$ replaced by 
$\overline{\dom} g$.	
\ref{p:vUperp!:ii}:	
Note that $\iota_U^* = \iota_{U^\perp}$ and
thus $\dom \iota_U^* = U^\perp$.
Hence \eqref{e:CQ} states exactly that $0\in \dom \iota_U^* + \dom g^*$.
It follows from \cite[Proposition~6.1(ii)~and~Corollary~6.5(i)]{MOR} that 
$v=\Pj{\overline{(U-\dom g)}\cap\overline{(U^\perp+\dom g^*)}}(0)$.
By \cref{prop:vD:vR} applied with 
$(C_1,C_2)$ replaced by $(\overline{\dom g},-\overline{\dom} g^*)$
we have 
\begin{equation}
v  = \Pj{\overline{U-\dom g}}(0).
\end{equation}
Now combine with \ref{p:vUperp!:i}.
\end{proof}
}

\textcolor{black}{
The result in \cref{p:vUperp!}\ref{p:vUperp!:ii} 
was first proved --- in an even more general form --- 
by Ryu, Liu, and Yin 
with a different argument relying on recession functions 
(see \cite[Lemma~3]{ucla}).
From now on, we assume:
\begin{empheq}[box=\mybluebox]{equation}
	\label{e:newassump:v}
v = \Pj{\overline{U-\dom g}}(0).
\end{empheq}
Note that \cref{e:newassump:v} holds if $X$ is
finite-dimensional by \cref{p:vUperp!}\ref{p:vUperp!:ii}.
In view of \cref{p:vUperp!}\ref{p:vUperp!:i}, we have 
\begin{equation}
	\label{e:vUperp!}
	v \in U^\perp.
\end{equation}
The fact that $v$ belongs to $U^\perp$ is new and crucial to our analysis. 
}

We now turn towards alternative descriptions of the set $Z$ of normal
solutions, defined in \eqref{e:Znonempty}. 
In passing, we mention that the next result is true even if $Z=\varnothing$.

\begin{proposition}
\label{0601b}
We have 
\begin{equation}
\label{e:Zvchar}
Z = \menge{x\in U}{U^\perp\cap\partial g(x-v)\neq\varnothing}
\end{equation}
and 
\begin{subequations}
\begin{align}
U \cap (v+\argmin g)
&\subseteq  
\menge{x\in U}{U^\perp\cap\partial g(x-v)\neq\varnothing}\label{e:0601b1}\\
&= \zer\big(\Nc{U}+\partial g(\cdot -v)\big)\label{e:0601b1.5}\\
&\subseteq 
U\cap(v+\dom \partial g)\cap 
\argmin(\iota_U+g(\cdot-v)) \label{e:0601b2}\\
&\subseteq
\argmin(\iota_U+g(\cdot-v)) \label{e:0601b3}\\
&\subseteq
U \cap (v+\dom g). \label{e:0601b4}
\end{align}
\end{subequations}
\end{proposition}
\begin{proof}
Recall that 
$v\in U^\perp$
by \eqref{e:vUperp!}. 
Hence $\Nc{U} = -v+\Nc{U}$. 
Now let $x\in X$.
Then  
\begin{subequations}
\begin{align}
x \in U\cap(v+\argmin g)
&\siff\big[x\in U \text{~and~}x-v\in \argmin g\big]\\
&\siff\big[x\in \zer\Nc{U} \text{~and~}0\in\partial g(x-v)\big]\\
&\siff\big[x\in \zer(-v+\Nc{U}) \text{~and~}0\in\partial g(x-v)\big]\\
&\siff\big[0\in -v+\Nc{U}(x) \text{~and~}0\in\partial g(x-v)\big]\\
&\RA 0\in -v+\Nc{U}(x)+\partial g(x-v)\\
&\siff
x\in Z\\
&\siff 
v\in \Nc{U}(x)+\partial g(x-v)\\
&\siff 
\big[x\in U \text{~and~} v\in U^\perp + \partial g(x-v)\big]\\
&\siff 
\big[x\in U \text{~and~} 0\in U^\perp + \partial g(x-v)\big]\\
&\siff 
\big[x\in U \text{~and~} U^\perp \cap \partial g(x-v)\neq \varnothing\big]\\
&\siff
x\in\zer\big(\Nc{U}+\partial g(\cdot -v)\big)\\
&\siff
0\in \big(\Nc{U}+\partial g(\cdot -v)(x),
\end{align}
\end{subequations}
which proves \cref{e:Zvchar}, \cref{e:0601b1}, and \cref{e:0601b1.5}.
Turning to \cref{e:0601b2},
let $x\in\zer(\Nc{U}+\partial g(\cdot-v))$.
On the one hand, $x\in\dom(\Nc{U}+\partial g(\cdot-v))$ and
thus $\Nc{U}(x)\neq\varnothing$ and $\partial g(x-v)\neq\varnothing$.
Hence $x\in U$ and $x-v\in \dom\partial g$, i.e., 
$x\in U\cap(v+\dom\partial g$. 
On the other hand, $\zer(\Nc{U}+\partial g(\cdot-v))
=\zer(\partial\iota_U+\partial g(\cdot-v))$.
Hence 
$0\in\partial\iota_U(x)+\partial g(\cdot-v)(x)
\subseteq\partial(\iota_U+g(\cdot-v))(x)$
and therefore $x$ minimizes $\iota_U+g(\cdot-v)$.
Finally, \cref{e:0601b3} and \cref{e:0601b4} are obvious. 
\end{proof}

\begin{example}[\bf linear-convex feasibility]
\label{e:lcf}
Suppose that $g=\iota_W$, 
where $W$ is a nonempty closed convex subset of $X$. 
Then $v=\Pj{\overline{U-W}}(0)$, 
$\argmin g = \dom\partial g = W$, and 
$v+\argmin g = v+W = v+\dom g$. 
Thus 
\cref{0601b} yields
\begin{equation}
\label{e:190519a}
Z = U\cap (v+V),
\end{equation}
a result that is well known (see \cite{Lukepaper}).
\end{example}

We are now ready for our first main result which 
provides a useful description of $Z$:

\begin{theorem}
\label{0601c}
Because $Z$ is nonempty, we have
\begin{equation}
\label{e:0601c}
Z = U\cap(v+\dom\partial g)\cap\argmin\big(\iota_U+g(\cdot-v)\big)
= \argmin\big(\iota_U+g(\cdot-v)\big).
\end{equation}
\end{theorem}
\begin{proof}
\cref{0601b} yields the inclusions
$Z \subseteq U\cap(v+\dom\partial g)\cap\\\
\big(\iota_U+g(\cdot-v)\big)
\subseteq \argmin\big(\iota_U+g(\cdot-v)\big)$.
Because $Z\neq\varnothing$, 
we let $x\in Z$, and also let $y\in\argmin(\iota_U+g(\cdot-v))\subseteq U$. 
First, by \eqref{e:Zvchar},
$x\in U$ and $U^\perp \cap \partial g(x-v)\neq\varnothing$. 
Secondly, it follows from \cref{0601a} (applied with $h=g(\cdot-v)$) 
that $U^\perp \cap\partial g(y-v)\neq\varnothing$.
Therefore, by using again \cref{e:Zvchar}, we obtain 
$y\in Z$.
\end{proof}

Here is an example of a case where $Z\neq \fady$.
\begin{example}
Suppose that $g$ is polyhedral. 
Then 
\cite[Theorem~5.6.1]{BBL97} implies that $U\cap (v+\dom g)=U\cap \dom g(\cdot-v)\neq \fady$.
Hence,
by \cite[Corollary~27.3(c)]{BC}
we have $Z=\argmin\big(\iota_U+g(\cdot-v)\big)$.
\end{example}

The underlying assumption that $Z$ be nonempty (see \cref{e:Znonempty}) 
in \cref{0601c} is critical:

\begin{example}
Suppose that $X=\RR^2$,
that $U=\{0\}\times\RR$
and that $g$ is the Rockafellar function defined by 
\begin{equation}
\label{e:def:rock:func}
g(\xi_1,\xi_2)
=
\begin{cases}
\max\{1-\sqrt{\xi_1},\abs{\xi_2}\},&\text{if}\ \xi_1\ge 0;
\\
+\infty,&\text{otherwise}.
\end{cases}
\end{equation}
(see \cite[Example~on~page~218]{Rock70}).
Then $v=0$ and it follows from \cite[Example~7.5]{MMW:2015} 
that 
$Z=\varnothing$,
$\argmin(\iota_U+g(\cdot-v))=\{0\}\times[-1,1]$,
and 
$U\cap(v+\dom\partial g)\cap\argmin(\iota_U+g(\cdot-v))=
\{0\}\times\{-1,1\}$. 
\end{example}
\begin{proof}
Clearly we have $U^\perp=\RR\times \{0\}$ and 
$\dom g=\RR_{+}\times \RR$.
Moreover,
\cite[Example~6.5]{MMW:2015} implies that
$\dom \partial g =\menge{(\xi_1,\xi_2)}{\xi_1>0, \xi_2\in \RR}
\cup \menge{(0,\xi_2)}{\xi_2 \ge 1}$,
and 
$\dom \partial g^*=\dom g^*=\menge{(\xi_1,\xi_2)}{\xi_1\le 0,\abs{\xi_2}\le 1 }$.
Therefore, using \cite[Corollary~6.5(i)]{MOR}
we learn that  
$v
=\Pj{(\overline{U-\dom}g)\cap (\overline{U^\perp+\dom}g^*)}(0)
=0$.
It follows from \cref{0601b} that
$Z=\menge{(0,\xi_2)}{U^\perp\cap \partial g((0,\xi_2))\neq \fady}$.
Now let $(0,\xi_2) \in U \cap \dom g$ and note that \cite[Example~6.5]{MMW:2015}
implies that
\begin{equation}
\partial g  (0,\xi_2)
=
\begin{cases}
\fady,&\text{if $\abs{\xi_2}< 1$};
\\
\RR_{-}\times \{1\},&\text{if $\abs{\xi_2}\ge 1$};
\\
\RR_{-}\times \{-1\},&\text{if $\abs{\xi_2}\le -1$},
\end{cases}
\end{equation}
which proves the claim that $Z=\fady$.
Finally, using \cref{e:def:rock:func}, 
we see that 
$\argmin(\iota_U+g(\cdot-v))=\argmin(\iota_U+g)=\{0\}\times[-1,1]$
and the conclusion follows.
\end{proof}

When $X=\RR$, then we obtain the following positive result,
which holds even when $Z=\varnothing$:

\begin{proposition}
\label{0601line}
Suppose that $X=\RR$. Then 
\begin{equation}
\label{e:0601line}
Z = U\cap(v+\dom\partial g)\cap\argmin\big(\iota_U+g(\cdot-v)\big). 
\end{equation}
More precisely, exactly one of the following cases holds:
\begin{enumerate}
\item 
$U=\{0\}$, $v=\Pj{-\cdom g}(0)$, $Z = 0\cdot\partial g(-v)$,
and either $\iota_U+g(\cdot-v)=\iota_{\{0\}}$ if $-v\in\dom g$
or $\iota_U+g(\cdot-v)=\iota_\varnothing$ if $-v\notin\dom g$. 
\item 
$U=\RR$, $v=0$, and $Z=\dom\partial g\cap \argmin g = \argmin g$. 
\end{enumerate}
\end{proposition}
\begin{proof}
Denote the right side of \cref{e:0601line} by $R$.
It is clear from \cref{0601b} that $Z\subseteq R$.
Now let $x\in R$. 
On the one hand, 
\begin{equation}
0\in\partial(\iota_U+g(\cdot-v))(x). 
\end{equation}
On the other hand,
$x\in\dom\partial \iota_U \cap \dom\partial g(\cdot-v)$. 
By the sum rule for the real line, we have 
\begin{equation}
\partial \iota_U(x) + \partial g(x-v) = \partial\big(\iota_U+g(\cdot-v)\big)(x).
\end{equation}
Altogether,
$0\in\partial \iota_U(x) + \partial g(x-v)$ and
thus $x\in Z$ by \cref{0601b}.
The remaining statements follow readily. 
\end{proof}

The previous results make it tempting to conjecture that
when $X=\RR$ and $Z=\varnothing$, then we have 
$\argmin(\iota_U+g(\cdot-v))=\varnothing$.
Unfortunately, this conjecture is false:

\begin{example}
Suppose that $X=\RR$, that $U=\{0\}$ and 
that $-\sqrt{x}$ with $\dom g=\RP$. 
Then $v=\Pj{-\cdom g}(0)=0$.
Hence $Z=\{0\}\cdot \partial g(0) = \varnothing$ by
\cref{0601line}
while $\argmin(\iota_U+g(\cdot-v))=\{0\}$
because $\iota_U+g(\cdot-v)=\iota_U+g = \iota_U = \iota_{\{0\}}$.
\end{example}

We conclude this section with another
useful consequence of \eqref{e:vUperp!}:

\begin{proposition}
We have $Z=\Pj{U}(F)$ and 
\begin{equation}
\label{e:Radu}
\Pj{U}\circ \Pj{F} = \Pj{Z}.
\end{equation}
\end{proposition}
\begin{proof}
Set $A = -v + \Nc{U}$ and $B=\partial g(\cdot-v)$, and 
note that  by \eqref{e:vUperp!}  $A=\Nc{U}$. 
Then the Douglas--Rachford operator corresponding to $(A,B)$ is
\cite[Proposition~3.2]{Sicon}
\begin{equation}
T(\cdot+v).
\end{equation}
Moreover $\J{A} := (\Id+A)^{-1} = \Pj{U}$.
Note that $A$ and $B$ are subdifferential operators, 
hence paramonotone by \cite[Theorem~2.2]{Iusem98}.
So \cite[Corollary~5.6]{74} yields
$F=Z+K$, $Z=\J{A}(F)=\Pj{U}(F)$, where $K := 
(\Id-\J{A^{-1}})(F) = \Pj{U^\perp}(F)\subseteq U^\perp$. 
Moreover, because $Z-Z \subseteq U$ and so $Z-Z \perp K$, 
we have
$\J{A}\Pj{Z+K}=\Pj{Z}$, 
equivalently, $\Pj{U}\Pj{F}=\Pj{Z}$,
by \cite[Theorem~6.7(ii)]{74}.
\end{proof}

\section{New dynamic results}
\label{sec:dyn}

Recall that 
\begin{equation}
T = \Id-\Pj{U} + \Pxg\R{U}.
\end{equation}

We start with a result that provides some
information on the shadow sequence $(\Pj{U}T^nx)_\nnn$.
(In passing, we note that only item \cref{l:190517a2} requires that $Z$ be nonempty.)

\begin{lemma}
\label{l:190517a}
Let $x\in X$. 
Then the following hold:
\begin{enumerate}
\item $\Pj{U}T^nx- \Px{g}\R{U}T^nx=T^nx-T^{n+1}x\to v \in U^\perp$. 
\label{l:190517a0}
\item $\Pj{U}T^nx-\Pj{U}\Pxg\R{U}T^nx = \Pj{U}T^nx-\Pj{U}T^{n+1}x \to 0$. 
\label{l:190517a1}
\item 
$-\Pj{U^\perp}\Pxg \R{U}T^nx  = \Pj{U^\perp}T^nx-\Pj{U^\perp}T^{n+1}x \to v$. 
\label{l:190517a1.1}
\item 
\label{l:190517a1.5}
All weak cluster points of $(\Pj{U}T^nx)_\nnn$ lie in 
$U\cap(v+\cdom g)$.
\item 
\label{l:190517a2}
The sequences $(nv+T^nx)_\nnn$, 
$(\Pj{U}T^nx)_\nnn$, and 
$(\Px{g}\R{U}T^nx)_\nnn$ 
are bounded. 
\end{enumerate}
\end{lemma}
\begin{proof}
\cref{l:190517a0}: Clear from the definition of $T$, \eqref{e:190517b} and \eqref{e:vUperp!}.
\cref{l:190517a1}:
Apply $\Pj{U}$ to  \cref{l:190517a0}.
\cref{l:190517a1.1}:
Apply $\Pj{U^\perp}$ to  \cref{l:190517a0}.
\cref{l:190517a1.5}: 
On the one hand, 
$(T^nx-T^{n+1}x) + \Px{g}\R{U}T^nx=\Pj{U}T^nx \in U$. 
On the other hand,
$\Px{g}\R{U}T^nx\in \dom\partial g
\subseteq \cdom g$.
Altogether, combined with \ref{l:190517a0},
we obtained the desired result.
\ref{l:190517a2}:
By \cref{Froof} and \cref{e:defF}, 
the sequence $(nv+T^nx)_\nnn$ is \fejer\ monotone with respect to 
$F \neq\varnothing$, hence it is bounded. 
Therefore, $(\Pj{U}T^nx)_\nnn = (\Pj{U}(nv+T^nx))_\nnn$ is also bounded.
The boundedness of $(\Px{g}\R{U}T^nx)_\nnn$ follows from \cref{l:190517a0}.
\end{proof}

Note that \cref{0601b} yields that $Z-v\subseteq (U-v)\cap \dom g$,
and thus $U-v\cap\dom g$ is nonempty. 
The next result provides information on function values of $g$ of
a sequence occurring in the Douglas--Rachford algorithm.

\begin{lemma}
\label{cake}
Let $x\in X$, 
let $y\in (U-v)\cap \dom g$,
and let $\nnn$. 
Then
\begin{subequations}
\label{E:cake0}
\begin{align}
g(y) &\geq g\big(\Pxg(\R{U}T^nx)\big) \\
&\qquad + \scal{y-\Pxg(\R{U}T^nx)}{(\Pj{U}T^nx-v)-\Pxg(\R{U}T^nx)} \\
&\qquad 
-\bscal{\Pj{U^\perp}T^nx-\Pj{U^\perp}T^{n+1}x-v}{\Pj{U^\perp}(nv+T^nx)}\\
&\qquad -(n+1){\scal{(\Id-T)T^nx-v}{0-v}}\\
&\geq g\big(\Pxg(\R{U}T^nx)\big) \\
&\qquad + \scal{y-\Pxg(\R{U}T^nx)}{(\Pj{U}T^nx-v)-\Pxg(\R{U}T^nx)} \\
&\qquad 
-\bscal{\Pj{U^\perp}T^nx-\Pj{U^\perp}T^{n+1}x-v}{\Pj{U^\perp}(nv+T^nx)}.
\end{align}
\end{subequations}
\end{lemma}
\begin{proof}
The characterization of the prox operator $\Pxg$ gives
\begin{equation}
\label{E:cake1}
g(y) \geq g\big(\Pxg(\R{U}T^nx)\big) 
+\scal{y-\Pxg(\R{U}T^nx)}{\R{U}T^nx-\Pxg(\R{U}T^nx)}. 
\end{equation}
We also have 
\begin{subequations}
\label{E:cake2}
\begin{align}
&
\mspace{-100mu}
\scal{y-\Pxg(\R{U}T^nx)}{\R{U}T^nx-\Pxg(\R{U}T^nx)}\\
&= \scal{y-\Pxg(\R{U}T^nx)}{\R{U}T^nx-(\Pj{U}T^nx-v)}\\
&\qquad + \scal{y-\Pxg(\R{U}T^nx)}{(\Pj{U}T^nx-v)-\Pxg(\R{U}T^nx)}\\
&= \scal{y-\Pxg(\R{U}T^nx)}{-\Pj{U^\perp}T^nx+v}\\
&\qquad + \scal{y-\Pxg(\R{U}T^nx)}{(\Pj{U}T^nx-v)-\Pxg(\R{U}T^nx)}.
\end{align}
\end{subequations}
Now write $y=u-v$, where $u\in U$. 
Then, using also the identity in \cref{l:190517a}\cref{l:190517a1.1} 
to derive \cref{e:cake1}, 
we have 
\begin{subequations}
\label{E:cake3}
\begin{align}
&
\mspace{-50mu}
\scal{y-\Pxg(\R{U}T^nx)}{-\Pj{U^\perp}T^nx+v}\\
&= 
\scal{(u-v)-\Pxg(\R{U}T^nx)}{-\Pj{U^\perp}T^nx+v}\\
&= 
\bscal{\underbrace{(u-\Pj{U}\Pxg(\R{U}T^nx)}_{\in U}-
\underbrace{(v+\Pj{U^\perp}\Pxg(\R{U}T^nx))}_{\in U^\perp}}%
{\underbrace{-\Pj{U^\perp}T^nx+v}_{\in U^\perp}}\\
&= 
\bscal{-v-\Pj{U^\perp}\Pxg(\R{U}T^nx)}{-\Pj{U^\perp}T^nx+v}\\
&= 
\bscal{-v+\Pj{U^\perp}T^nx-\Pj{U^\perp}T^{n+1}x}{-\Pj{U^\perp}T^nx+v}\label{e:cake1}\\
&= 
-\bscal{\Pj{U^\perp}T^nx-\Pj{U^\perp}T^{n+1}x-v}{\Pj{U^\perp}T^nx-v}\\
&= 
-\bscal{\Pj{U^\perp}T^nx-\Pj{U^\perp}T^{n+1}x-v}{\Pj{U^\perp}(nv+T^nx)-(n+1)v}\\
&= 
-\bscal{\Pj{U^\perp}T^nx-\Pj{U^\perp}T^{n+1}x-v}{\Pj{U^\perp}(nv+T^nx)}\\
&\qquad -(n+1)\bscal{\Pj{U^\perp}T^nx-\Pj{U^\perp}T^{n+1}x-v}{-v}\\
&= 
-\bscal{\Pj{U^\perp}T^nx-\Pj{U^\perp}T^{n+1}x-v}{\Pj{U^\perp}(nv+T^nx)}\\
&\qquad -(n+1)\scal{T^nx-T^{n+1}x-v}{-v}\\
&= 
-\bscal{\Pj{U^\perp}T^nx-\Pj{U^\perp}T^{n+1}x-v}{\Pj{U^\perp}(nv+T^nx)}\\
&\qquad -(n+1)\underbrace{\scal{(\Id-T)T^nx-v}{0-v}}_{\leq 0 \text{~by \cref{e:defv}}}\\
&\geq 
-\bscal{\Pj{U^\perp}T^nx-\Pj{U^\perp}T^{n+1}x-v}{\Pj{U^\perp}(nv+T^nx)}.
\end{align}
\end{subequations}
Therefore, 
substituting \cref{E:cake2} and \cref{E:cake3}
into \cref{E:cake1}, we obtain 
\begin{subequations}
\begin{align}
g(y) &\geq g\big(\Pxg(\R{U}T^nx)\big) \\
&\qquad + \scal{y-\Pxg(\R{U}T^nx)}{(\Pj{U}T^nx-v)-\Pxg(\R{U}T^nx)} \\
&\qquad 
-\bscal{\Pj{U^\perp}T^nx-\Pj{U^\perp}T^{n+1}x-v}{\Pj{U^\perp}(nv+T^nx)}\\
&\qquad -(n+1)\underbrace{\scal{(\Id-T)T^nx-v}{0-v}}_{\leq 0}\\
&\geq g\big(\Pxg(\R{U}T^nx)\big) \\
&\qquad + \scal{y-\Pxg(\R{U}T^nx)}{(\Pj{U}T^nx-v)-\Pxg(\R{U}T^nx)} \\
&\qquad 
-\bscal{\Pj{U^\perp}T^nx-\Pj{U^\perp}T^{n+1}x-v}{\Pj{U^\perp}(nv+T^nx)},
\end{align}
\end{subequations}
which completes the proof. 
\end{proof}

We are now able to locate weak cluster points
of the shadow sequence $(\Pj{U}T^nx)_\nnn$:

\begin{lemma}
\label{kuchen}
Let $x\in X$ and 
let $y\in (U-v)\cap \dom g$. 
Then there exists a sequence $(\varepsilon_n)_\nnn$ in $\RR$ 
such that
\begin{equation}
\varepsilon_n\to 0
\end{equation}
and
for every $\nnn$, we have 
\begin{subequations}
\label{E:kuchen0}
\begin{align}
g(y) &\geq g\big(\Pxg(\R{U}T^nx)\big) + \varepsilon_n
+(n+1)\bscal{T^nx-T^{n+1}x-v}{v}\\
&\geq g\big(\Pxg(\R{U}T^nx)\big) + \varepsilon_n.
\end{align}
\end{subequations}
Moreover, the sequence
\begin{equation}
\label{E:kuchen10}
\text{$(\Pxg(\R{U}T^nx))_\nnn$ is bounded,
all its weak cluster points are minimizers of $\iota_{U-v}+g$, }
\end{equation}
\begin{equation}
\label{E:kuchen3}
g\big(\Pxg(\R{U}T^nx)\big) \to \inf g(U-v),
\end{equation}
and 
\begin{equation}
\label{E:kuchen4}
(n+1)\bscal{T^nx-T^{n+1}x-v}{v}\to 0. 
\end{equation}
Finally, the sequence
\begin{equation}
\label{E:kuchen9}
\text{$(\Pj{U}T^nx)_\nnn$ is bounded and all its weak cluster points 
are minimizers of $\iota_U+g(\cdot-v)$. }
\end{equation}
\end{lemma}
\begin{proof}
\cref{l:190517a}\cref{l:190517a2}\&\cref{l:190517a0}
yield that $(y-\Pxg\R{U}T^nx)_\nnn$ is bounded 
and that $\Pj{U}T^nx-v-\Pxg\R{U}T^nx\to 0$.
Thus
\begin{equation}
\label{E:kuchen1}
\scal{y-\Pxg\R{U}T^nx}{(\Pj{U}T^nx-v)-\Pxg(\R{U}T^nx)}\to 0.
\end{equation}
\cref{l:190517a}\cref{l:190517a1.1}\&\cref{l:190517a0}
yield that $\Pj{U^\perp}T^nx-\Pj{U^\perp}T^{n+1}x-v\to 0$ and
that $(\Pj{U^\perp}(nv+T^nx))_\nnn$ is bounded. 
Hence 
\begin{equation}
\label{E:kuchen2}
-\bscal{\Pj{U^\perp}T^nx-\Pj{U^\perp}T^{n+1}x-v}{\Pj{U^\perp}(nv+T^nx)}\to 0.
\end{equation}
Setting
\begin{align}
\varepsilon_n 
&= \scal{y-\Pxg\R{U}T^nx)}{(\Pj{U}T^nx-v)-\Pxg(\R{U}T^nx)}\\
&\qquad -\bscal{\Pj{U^\perp}T^nx-\Pj{U^\perp}T^{n+1}x-v}{\Pj{U^\perp}(nv+T^nx)},
\end{align}
we see that \cref{E:kuchen0} is a consequence of \cref{cake},
\cref{E:kuchen1} and \cref{E:kuchen2}. 

By \cref{l:190517a}\cref{l:190517a2}, $(\Pxg\R{U}T^nx)_\nnn$ is bounded.
Let $c$ be a weak cluster point of $(\Pxg\R{U}T^nx)_\nnn$,
say $\Pxg\R{U}T^{k_n}x\weakly c$. 
\cref{l:190517a}\cref{l:190517a0} implies that 
\begin{equation}
\label{E:kuchen6}
\Pxg\R{U}T^{k_n}x\weakly c\in U-v. 
\end{equation}
Now abbreviate 
$\alpha_n = (n+1)\bscal{T^nx-T^{n+1}x-v}{v}$. 
Then \cref{E:kuchen0}
yields
\begin{equation}
\label{E:kuchen5}
g(y) \geq g\big(\Pxg(\R{U}T^nx)\big) + \varepsilon_n+ \alpha_n
\geq g\big(\Pxg(\R{U}T^nx)\big) + \varepsilon_n.
\end{equation}
The weak lower semicontinuity of $g$ now yields 
\begin{equation}
\label{E:kuchen7}
g(y)\geq \varlimsup g(\Pxg(\R{U}T^{k_n}x))\geq
\varliminf g(\Pxg(\R{U}T^{k_n}x))\geq g(c).
\end{equation}
Combining with \cref{E:kuchen6}, we deduce that
\begin{equation}
c \in (U-v)\cap \dom g.
\end{equation}
Set $\mu = \inf g(U-v)$. 
Choosing $y= c$ in \cref{E:kuchen7} yields
\begin{equation}
\label{E:kuchen8}
g(\Pxg(\R{U}T^{k_n}x)) \to g(c)\geq \mu.
\end{equation}
Now choosing $y$ so that $g(y)$ is as close to $\mu$ as we like,
we deduce from \cref{E:kuchen7} and \cref{E:kuchen8} that
\begin{equation}
g(\Pxg(\R{U}T^{k_n}x)) \to g(c)= \mu.
\end{equation}
Hence $c$ is a minimizer of $\iota_{U-v}+g$. 
Because $c$ was an \emph{arbitrary} weak cluster point of 
$(\Pxg\R{U}T^nx)_\nnn$, we obtain through a simple proof by contradiction that 
\begin{equation}
g(\Pxg(\R{U}T^{n}x)) \to \mu,
\end{equation}
i.e., \cref{E:kuchen3} holds. 

Next, \cref{E:kuchen5} with $y= c$ yields
$\mu = g(c)\geq \mu+\varlimsup\alpha_n
\geq \mu+\varliminf\alpha_n\geq \mu$.
Thus $\alpha_n\to 0$ and
\cref{E:kuchen4} follows.

Finally, \cref{E:kuchen9} follows from
\cref{E:kuchen10} and \cref{l:190517a}\ref{l:190517a0}.
\end{proof}

\begin{remark}
Note that \eqref{E:kuchen4} is equivalent to 
$n\cdot\scal{T^nx-T^{n+1}x-v}{v}\to 0$. 
On the other hand,
\eqref{Froof2} and \eqref{Froof4} combined with \cite[Chapter~III, Section~14, Theorem on p.~124]{Knopp} 
(or \cite[Problem~3.2.35]{KN}) yields
${n}\cdot\|T^nx-T^{n+1}x-v\|^2\to 0$. 
We do not know whether 
$n\cdot\|T^nx-T^{n+1}x-v\|\to 0$.
\end{remark}
\section{The main result}
\label{sec:main}

We are now ready for the main result.
\textcolor{black}{In the following we set
\begin{equation}
\label{e:def:y_x}
y\colon X\to X\colon x\mapsto \lim_{n\to\infty}\Pj{F}(nv+T^nx),
\end{equation}
which is well defined by \cref{Froof}.
}
\begin{theorem}[{\bf main result}]
\label{main}
Let $x\in X$.
Then 
\begin{equation}
\label{e:main}
\Pj{U}T^nx \weakly \Pj{U}y(x) \in \argmin(\iota_U+g(\cdot-v)), 
\end{equation}
$T^{n+1}x-T^nx+\Pj{U}T^nx=\Pxg(\R{U}T^nx) \weakly -v+\Pj{U}y(x)$, and 
\begin{equation}
g(\Pxg\R{U}T^nx)\to\min(\iota_U+g(\cdot-v)).
\end{equation}
\end{theorem}
\begin{proof}
For brevity, we write $y=y(x)$. 	
Because $\Pj{U}$ is continuous, we have 
\begin{equation}
\Pj{U}\Pj{F}(nv+T^nx)\to \Pj{U}y.
\end{equation}
On the other hand,
$\Pj{U}\Pj{F} = \Pj{Z} = \Pj{Z}\Pj{U}$ 
by \cref{e:Radu} and \cref{e:restproj}.
Invoking the fact that $v\in U^\perp$ (see \cref{e:vUperp!}),
we conclude altogether that
\begin{equation}
\label{e:main1}
\Pj{Z}\Pj{U}T^nx = \Pj{Z}\Pj{U}(nv+T^nx)\to \Pj{U}y.
\end{equation}
Recall from \eqref{E:kuchen9} and \eqref{e:0601c} that 
$(\Pj{U}T^nx)_\nnn$ is bounded and that 
all its cluster points lie in $\argmin(\iota_U+g(\cdot-v)) = Z$.
Now let $z$ be an arbitrary weak cluster point of $(\Pj{U}T^nx)_\nnn$,
say $\Pj{U}T^{k_n}x\weakly z\in Z\subseteq U$.
Then $\Pj{Z}\Pj{U}T^{k_n}x\weakly \Pj{Z}z=z$ using \eqref{e:PZweakly}. 
Combining with \cref{e:main1}, 
we deduce that $z = \Pj{U}y$.
Hence \emph{every} weak cluster point of $(\Pj{U}T^nx)_\nnn$ 
coincides with $\Pj{U}y$.
In view of the boundedness of $(\Pj{U}T^nx)_\nnn$,
we obtain \eqref{e:main}.
The remainder follows from
\cref{l:190517a}\cref{l:190517a0} and 
\eqref{E:kuchen3}.
\end{proof}

\begin{example}[\bf linear-convex feasibility]
Suppose that $g=\iota_W$, 
where $W$ is a nonempty closed convex subset of $X$
such that $U\cap(v+W)\neq \fady$.
Then, $0\in \dom g^*$ which implies that 
$0\in U^\perp+\dom g^*$, hence \cref{e:CQ} is verified.
Moreover, 
$v=\Pj{\overline{U-W}}(0)$ by \cite[Proposition~3.16]{Sicon}
and $(\forall x\in X)$
$\Pj{U}T^nx \weakly \Pj{U}y \in U\cap(v+W)$,
 where $y = \lim_{n\to\infty}\Pj{F}(nv+T^nx)$
 by \cref{main}.
\end{example}

\begin{example}
\label{ex:good:II}
Suppose that $W$ is a linear subspace of $X$ such that 
$\{0\}\subsetneqq W \subsetneqq U^\perp$.
Let $w\in W\smallsetminus\{0\}$, 
let $b\in (U^\perp\cap W^{\perp})\smallsetminus\{0\}$,
and suppose that
$g=\tfrac{1}{2}\norm{\cdot}^2+\scal{w}{\cdot}+\iota_{-b+W}$.
Let $x\in X$. 
Then the following hold:
\begin{enumerate}
\item 
\label{ex:good:II:-i}
$\partial g= w+\Id+\Nc{-b+W}$. 
\item 
\label{ex:good:II:0}
$U\cap W=\{0\}$. 
\item 
\label{ex:good:II:i}
$\dom g = \dom \partial g = -b+W$, 
$\dom g^*=X$, and $0\in U^\perp+\dom g^*=X$.
\item
\label{ex:good:II:ii}
$v=b\in U^\perp\cap W^{\perp}$. 
\item
\label{ex:good:II:iii}
$-v+\Nc{U}=\Nc{U}$.
\item
\label{ex:good:II:vii}
$Z=\{0\}$.
\item
\label{ex:good:II:v-}
$\Pxg = -b -\tfrac{1}{2}w+\tfrac{1}{2}\Pj{W}$.
\item
\label{ex:good:II:v}
$T=-b-\tfrac{1}{2}w+\Id-\Pj{U}-\tfrac{1}{2}\Pj{W}$.
\item
\label{ex:good:II:vi}
$F=U^\perp\cap(-w+W^\perp)$.
\item
\label{ex:good:II:vi.5}
$0\notin F$. 
\item
\label{ex:good:II:vi+}
$(\forall n\geq 1)$
$T^nx = (\Pj{U^\perp}-(1-\tfrac{1}{2^n})\Pj{W})x-nb -(1-\tfrac{1}{2^n})w$. 
\item
\label{ex:good:II:vi++}
$(\forall n\geq 1)$
$\Pj{U}T^nx = 0$. 
\end{enumerate}
\end{example}
\begin{proof}
Note that $U+W \subsetneqq U+U^\perp = X$ and thus 
$U^\perp \cap W^\perp =(U+W)^\perp \supsetneqq \{0\}$.
Hence the choice of $b$ is possible. 
\ref{ex:good:II:-i}: Clear. 
\ref{ex:good:II:0}:
Indeed, 
$\{0\}\subseteq U\cap W\subseteq U\cap U^\perp=\{0\}$.
\ref{ex:good:II:i}:
It is clear that $\dom g = \dom \partial g = -b+W$.
Because $\lim_{\|x\|\to+\infty} g(x)/\|x\|=+\infty$,
it follows that $\dom g^*=\dom \partial g^*=X$ by, e.g.,
\cite[Proposition~14.15 and Proposition~16.27]{BC}.
\ref{ex:good:II:ii}:
Using \eqref{e:vUperp!} and \ref{ex:good:II:i},
we obtain
$v
=\Pj{\overline{U-\dom g}}(0)
=\Pj{b+U+W}(0)
=b+\Pj{U+W}(0-b)
=\Pj{(U+W)^\perp}(b)
=\Pj{U^\perp\cap W^{\perp}}(b)=b$.
\ref{ex:good:II:iii}:
Clear from \ref{ex:good:II:ii}.
\ref{ex:good:II:vii}:
This follows from \eqref{e:Znonempty},
\ref{ex:good:II:-i}, \ref{ex:good:II:0},
and \ref{ex:good:II:i}.
\ref{ex:good:II:v-}:
Set $y=-b-\tfrac{1}{2}w+\tfrac{1}{2}\Pj{W}x$.
Then $y\in -b+W$.
Thus,
$\Pj{W^\perp}x \in -2b+W^\perp$
$\siff$
$x\in 2(-b-\tfrac{1}{2}w+\tfrac{1}{2}\Pj{W}x)+w+W^\perp
= 2y+w+W^\perp = y+w+y+\Nc{-b+W}(y) = (\Id+\partial g)(y)$
$\siff$
$y=\Pxg(x)$.
\ref{ex:good:II:v}:
This follows from \eqref{e:defT} and \ref{ex:good:II:v-}. 
\ref{ex:good:II:vi}:
Using \eqref{e:defF} and \ref{ex:good:II:v}, we obtain
$x\in F$
$\siff$
$x=T(x+v)=T(x+b)$
$\siff$
$x=-b-\tfrac{1}{2}w+x+b-\Pj{U}(x+b)-\tfrac{1}{2}\Pj{W}(x+b)$
$\siff$
$0=\tfrac{1}{2}w+\tfrac{1}{2}\Pj{U}x+\tfrac{1}{2}\Pj{W}x$
$\siff$
[$x\in U^\perp$ and $x\in -w+W^\perp$].
\ref{ex:good:II:vi.5}:
We have the equivalences
$0\in F$
$\siff$
$0=T(0+v)$
$\siff$
$0=T(b)$
$\siff$
$0=-b-\tfrac{1}{2}w+b-\Pj{U}b-\tfrac{1}{2}\Pj{W}b$
$\siff$
$0=-\tfrac{1}{2}w$, which is absurd. 
\ref{ex:good:II:vi+}:
This follows from \ref{ex:good:II:vi} and induction.
\ref{ex:good:II:vi++}:
Clear from \ref{ex:good:II:vi+}.
\end{proof}

\begin{remark}
\label{r:whatsnew}
We point out that in \cite[Theorem~4.4]{BM:MPA17}
the authors provide an instance where the shadow sequence converges.
The proof in \cite{BM:MPA17} 
critically relies on the assumption 
that $Z\subseteq F$. 
Our new result does not require 
this assumption. 
Indeed, by 
\cref{ex:good:II}\ref{ex:good:II:vii}\&\ref{ex:good:II:vi.5}, 
$Z=\{0\}$ and $Z\cap F=\varnothing$.
\end{remark}

\begin{example}
\label{ex:bad:cq:fail}
Suppose that $X$ is finite-dimensional\footnote{
 We require this assumption in the proof of 
item~\ref{ex:bad:cq:fail:iii} which relies on \cite{MOR}.
}, that $U\neq\{0\}$, 
let $u^*\in U\smallsetminus \{0\}$, 
suppose that\footnote{Given a nonempty closed convex subset $C$ of $X$, 
the associated distance function to the set $C$ is denoted by 
$\dist{C}$.}
$g=\tfrac{1}{2}\dist{U}^2+\scal{u^*}{\cdot}$,
and let $x\in X$. 
Then the following hold:
\begin{enumerate}
\item 
\label{ex:bad:cq:fail:0}
$\partial g = \nabla g = u^*+\Pj{U^\perp}$. 
\item
\label{ex:bad:cq:fail:i}
$U-\dom \grad g=U-\dom g = X$.
\item
\label{ex:bad:cq:fail:ii}
$\ran\Nc{U}+\ran\partial g 
= U^\perp +\dom g^*=U^\perp+\dom \partial g^* = u^*+U^\perp$ is closed.
\item
\label{ex:bad:cq:fail:ii.5}
$0\not\in \overline{U^\perp +\dom g^*} = \overline{\ran \Nc{U}+\ran\partial g}$. 
\item
\label{ex:bad:cq:fail:iii}
$v=u^*\in U\smallsetminus\{0\}$.
\item
\label{ex:bad:cq:fail:iv}
$Z=U$.
\item
\label{ex:bad:cq:fail:iii.5}
$\Pxg = -u^*+\Id-\tfrac{1}{2}\Pj{U^\perp}$.
\item
\label{ex:bad:cq:fail:vi}
$T=\Pxg = -u^*+\Id-\tfrac{1}{2}\Pj{U^\perp}$.
\item
\label{ex:bad:cq:fail:vi.5}
$F=U$. 
\item
\label{ex:bad:cq:fail:vii}
$(\forall\nnn)$
$T^n x =-nu^*+\Pj{U}x+\tfrac{1}{2^n}\Pj{U^\perp}x$.
\item
\label{ex:bad:cq:fail:viii}
$(\forall\nnn)$
$\Pj{U}T^n x =-nu^*+\Pj{U}x$.
\item
\label{ex:bad:cq:fail:viii.5}
$(\forall\nnn)$
$\norm{T^nx} \geq \norm{\Pj{U}T^n x}\ge n\norm{u^*}-\norm{\Pj{U}x}\to +\infty$.
\end{enumerate}
\end{example}
\begin{proof}
\ref{ex:bad:cq:fail:0}:
Clear since $\nabla \tfrac{1}{2}\dist{U}^2 = \Id-\Pj{U} = \Pj{U^\perp}$. 
Note that $\grad g=u^*+\Id-\Pj{U}=u^*+\Pj{U^\perp}$.
\ref{ex:bad:cq:fail:i}:
$U-\dom\partial g = U-X=X$.
\ref{ex:bad:cq:fail:ii}:
$\dom \partial g^*
=\ran \nabla g 
=u^*+U^\perp$ is closed.
On the other hand, $\dom \partial g^*$ is a dense
subset of $\cdom g^*$.
Hence $\dom\partial g^*=\dom g^*=u^*+U^\perp$
and thus $\ran \Nc{U} + \ran\partial g = U^\perp + (u^*+U^\perp) = u^*+U^\perp$.
\ref{ex:bad:cq:fail:ii.5}:
Clear from \ref{ex:bad:cq:fail:ii}
and the assumption that $u^*\neq 0$. 
\ref{ex:bad:cq:fail:iii}:
By \cite[Proposition~6.1]{MOR}, \ref{ex:bad:cq:fail:i}, 
and \ref{ex:bad:cq:fail:ii},
we have $v
=\Pj{\overline{U-\dom g}\ \cap\ \overline{U^\perp+\dom g^*}}(0)
=\Pj{u^*+U^\perp}(0)
=u^*+\Pj{U^\perp}(0-u^*)
=\Pj{U}(u^*)=u^*
$.
\ref{ex:bad:cq:fail:iv}:
Using \eqref{e:Znonempty}, \ref{ex:bad:cq:fail:0},
and \ref{ex:bad:cq:fail:iii}, we have
$x\in Z$
$\siff$
$v\in \Nc{U}(x)+\partial g(x-v)$
$\siff$
[$x\in U$ and $u^*\in U^\perp+u^*+\Pj{U^\perp}(x-u^*)$]
$\siff$
$x\in U$.
\ref{ex:bad:cq:fail:iii.5}:
Set $y=-u^*+x-\tfrac{1}{2}\Pj{U^\perp}x$.
By \ref{ex:bad:cq:fail:0} and \ref{ex:bad:cq:fail:iii},
$y+\nabla g(y) = 
(-u^*+x-\tfrac{1}{2}\Pj{U^\perp}x)
+(u^*+\Pj{U^\perp}(-u^*+x-\tfrac{1}{2}\Pj{U^\perp}x))
= x$. 
Thus $y=\Pxg(x)$ as claimed. 
\ref{ex:bad:cq:fail:vi}:
Using \eqref{e:defT} and \ref{ex:bad:cq:fail:iii.5},
we obtain $T
=\Id-\Pj{U}+\Pxg \R{U}
=\Pj{U^\perp}+\Pxg(\Pj{U}-\Pj{U^\perp})
=\Pj{U^\perp}-u^*+(\Id-\tfrac{1}{2}\Pj{U^\perp})
 (\Pj{U}-\Pj{U^\perp})
 =-u^*+\Pj{U}+\tfrac{1}{2}\Pj{U^\perp}
 =-u^*+\Pj{U}+\Pj{U^\perp}-\tfrac{1}{2}\Pj{U^\perp}
 =-u^*+\Id-\tfrac{1}{2}\Pj{U^\perp}
 =\Pxg
 $.
\ref{ex:bad:cq:fail:vi.5}:
Using \eqref{e:defF}, \ref{ex:bad:cq:fail:iii}, and 
\ref{ex:bad:cq:fail:vi}, we have
$x\in F$
$\siff$
$x=T(x+v)$
$\siff$
$x=-u^*+\Pj{U}x+\tfrac{1}{2}\Pj{U^\perp}(x+v)$
$\siff$
$x=\Pj{U}x+\tfrac{1}{2}\Pj{U^\perp}x$
$\siff$
$x\in U$. 
\ref{ex:bad:cq:fail:vii}:
This follows from \ref{ex:bad:cq:fail:vi} and
\ref{ex:bad:cq:fail:iii} by a straight-forward induction.
\ref{ex:bad:cq:fail:viii}:
Apply $\Pj{U}$ to \ref{ex:bad:cq:fail:vii} and use 
\ref{ex:bad:cq:fail:iii}.
\ref{ex:bad:cq:fail:viii.5}:
This follows from \ref{ex:bad:cq:fail:viii}.
\end{proof}

\begin{remark}
\cref{ex:bad:cq:fail} illustrates the importance
of the constraint qualification \eqref{e:CQ};
indeed, it provides a scenario where
\eqref{e:CQ} fails (see item \ref{ex:bad:cq:fail:ii.5}) and the
shadow sequence never converges (see item \ref{ex:bad:cq:fail:viii.5}). 
\end{remark}

\begin{remark}
While \cref{main} guarantees that $(\Pj{U}T^nx)_\nnn$ 
converges weakly to a minimizer of $\iota_U+g(\cdot-v)$, 
we leave numerical experiments and the 
development of 
meaningful termination criteria as topics for future research.
A promising starting point appears to be the analysis in 
\cite[Section~5]{Banjac}. 
\end{remark}

	The remaining results in this section were 
	inspired by a referee's question.
\begin{theorem}[\bf switching the order of the operators]
	\label{thm:switch}
Set $\widetilde{T}=\Id-\Pxg+\Pj{U}\R{g}=\Id-\Pxg+\Pj{U}(2\Pxg-\Id)$.
Suppose that\footnote{This assumption is satisfied if, for instance,
$X$ is finite-dimensional.
To see this, proceed as in the proof
of \cref{p:vUperp!}\ref{p:vUperp!:ii},
with the roles of $\iota_U $ and $g$ switched.
}
 $\Pj{\overline{\ran}(\Id -\widetilde{T})}(0)=-v$.
Let $x\in X$.
Then the following hold:
\begin{enumerate}
	\item 
	\label{thm:switch:i}
	$(\forall \nnn)$ $\Pj{U}\widetilde{T}^n=\Pj{U}T^n\R{U}$.
		\item 
	\label{thm:switch:i:ii}
    $\widetilde{T}^nx-\widetilde{T}^{n+1}x
    =\Pxg\widetilde{T}^nx-2\Pj{U}\Pxg\widetilde{T}^nx+\Pj{U}\widetilde{T}^nx
    =\Pj{U}\widetilde{T}^nx-\R{U}\Pxg\widetilde{T}^nx\to -v$.
	\item 
	\label{thm:switch:ii}
	$\Pj{U}\widetilde{T}^nx-\Pj{U}\Pxg\widetilde{T}^n x\to\Pj{U}(-v)=0$.
	\item
	\label{thm:switch:iii}
	$\Pj{U}T^n x\weakly\Pj{U}y(x) \in \argmin (\iota_U+g(\cdot-v))$.
	\item
	\label{thm:switch:iv}
	$\Pj{U}\widetilde{T}^n x\weakly\Pj{U}y(\R{U}x) \in \argmin (\iota_U+g(\cdot-v))$.
	\item
	\label{thm:switch:v}
    $\Pxg\widetilde{T}^n x\weakly\Pj{U}y(\R{U}x)-v\in \dom g$.
\end{enumerate}		
\end{theorem}
\begin{proof}
Observe that $\Pj{U}\R{U}=\Pj{U}$ and $\R{U}^2=\Id$.
\ref{thm:switch:i}:
Using \cite[Theorem~2.7(i)]{BM:order15} we learn that
$(\forall \nnn)$
$\Pj{U}\widetilde{T}^n
=\Pj{U}\R{U}\widetilde{T}^n\R{U}\R{U}
=\Pj{U}T^n\R{U}$.
\ref{thm:switch:i:ii}:
$\widetilde{T}^n-\widetilde{T}^{n+1}
=\Pxg\widetilde{T}^n-\Pj{U}\R{g}\widetilde{T}^n
=\Pxg\widetilde{T}^n-2\Pj{U}\Pxg\widetilde{T}^n+\Pj{U}\widetilde{T}^n
=\Pj{U}\widetilde{T}^n-\R{U}\Pxg\widetilde{T}^n$.
Now combine with \cref{e:190517b}.
\ref{thm:switch:ii}:
Recall that $-v\in U^\perp$ by \cref{e:vUperp!}.
Now combine with \ref{thm:switch:i:ii}.
\ref{thm:switch:iii}:
This is \cref{main}.
\ref{thm:switch:iv}:
Combine \ref{thm:switch:i}
and \ref{thm:switch:iii} with $x$
replaced by $\R{U}x$.
\ref{thm:switch:v}:
It follows from \ref{thm:switch:ii} and 
\ref{thm:switch:iv}
that 
$\Pj{U}\Pxg\widetilde{T}^n x\weakly\Pj{U}y(\R{U}x)$.
Now combine with \ref{thm:switch:i:ii}.
\end{proof}

\textcolor{black}{
In the setting of \cref{main}, we point out 
that no general conclusion can be drawn about the sequence
$(\Pxg T^n x)_\nnn$ as we illustrate below. 
\begin{example}[\bf $(\Pxg T^n x)_\nnn$ may converge]
	Suppose that $(U, g)=(X, \iota_X)$.
	Then $\Pj{U}=\Pxg=T=\widetilde{T}=\Id$.
	Hence, $\ran(\Id-T)=\ran(\Id-\widetilde{T})=\{0\}$.
	Consequently, $v=-v=0$ and 
	$(\forall \nnn)$ $(\forall x\in X)$ 
	$\Pxg T^nx=x=\lim_{n\to \infty}\Pxg T^nx$.
\end{example}	
\begin{example}[\bf $(\Pxg T^n x)_\nnn$ may have no cluster points]
	Suppose that $X=\RR^2$, that $U=\RR\times \{0\}$,
	that $C=\epi(\abs{\cdot}+1)$ and that $g=\iota_C$.
	Let $x\in \left[-1,1\right]\times \{0\}$.
	Using induction, one can show that 
	 $(\forall n\in \{1,2,\ldots\})$ $T^n x=(0,n)\in C$.
	Consequently, $\norm{\Pxg T^n x}=\norm{\Pj{C} T^n x}=n\to +\infty$.
\end{example}	
}

\section{Minimizing the sum of finitely many functions}
\label{sec:app}

In this section we assume for simplicity 
that 
\begin{empheq}[box=\mybluebox]{equation}
\text{$X$ is finite-dimensional,} 
\end{empheq}
that $m\in \{2,3,\ldots\}$,
that $I=\{1,2,\ldots, m\}$,
 and that
\begin{empheq}[box=\mybluebox]{equation}
\text{$g_i\colon X\to \left]-\infty,+\infty\right]$
is convex, lower semicontinuous, and proper,}
\end{empheq}
for every $i\in I$.
Furthermore, we set (see also \cite{BC} and \cite{Comb09})
\begin{equation}
\begin{cases}
&\bX=\bigoplus_{i\in I} X,
\\
&\bg=\bigoplus_{i\in I}g_i,
\\
&{\bDelta}=\menge{(x,x,\ldots,x)\in \bX}{x\in X},
\\
&\bZ=\menge{\bx\in \bX }{\bv\in N_\bDelta(\bx)+\partial \bg (\bx-\bv)},
\\
& (\forall i\in I) \quad D_i=\overline{\dom}  g_i,
\\
&\bD=\bigtimes_{i\in I} D_i,
\\
&\bv=(v_i)_{i\in I}=\Pj{\cran(\Id-\bT)} (\bzero),
\\
&\bT=\Id-\Pj{\bDelta}+\Px{\bg}\R{\bDelta},
\\
&\bj\colon X\to \bDelta \colon x\mapsto (x,x,\ldots,x),
\\
&\be \colon \bX\to X\colon (x_i)_{i\in I}\mapsto \tfrac{1}{m}\big(\sum_{i\in I}x_i\big).
\end{cases}
\end{equation}
\begin{remark}
In passing we point out that, by
\cite[Theorem~2.16]{BMWnear},
we have $(\forall i \in I)$ 
$D_i=\overline{\dom}\ \partial g_i=\overline{\dom}\ g_i$.
\end{remark}

\begin{fact}
\label{prop:prod:sp}
Write $\bx=(x_i)_{i\in I}\in \bX$.
Then the  following hold: 
\begin{enumerate}
\item
\label{prop:prod:sp:0}
$\bg \colon \bX\to \left]-\infty,+\infty\right]$ 
is convex, lower semicontinuous, and proper.
\item
\label{prop:prod:sp:i}
$\bg^*=\bigoplus_{i\in I}g_i^*$.
\item
\label{prop:prod:sp:ii}
$\partial \bg=\bigtimes_{i\in I}\partial g_i$.
\item
\label{prop:prod:sp:ii.5}
$\Pj{\bDelta}\bx=\bj\big(\tfrac{1}{m}\sum_{i\in I}x_i\big)$.
\item
\label{prop:prod:sp:iii}
$\Px{\bg}=\bigtimes_{i\in I}\Px{g_i}$.
\item  
\label{prop:prod:sp:iv}
$\bDelta^\perp=\menge{\bu\in \bX}{\sum_{i\in I}u_i=0}$.
\end{enumerate}
\end{fact}
\begin{proof}
\ref{prop:prod:sp:0}: 
Clear.
\ref{prop:prod:sp:i}:
This is \cite[Proposition~13.30]{BC}.
\ref{prop:prod:sp:ii}:
This is \cite[Proposition~16.9]{BC}.
\ref{prop:prod:sp:ii.5}:
This is \cite[Proposition~26.4(ii)]{BC}.
\ref{prop:prod:sp:iii}:
This is \cite[Proposition~24.11]{BC}.
\ref{prop:prod:sp:iv}:
This is \cite[Proposition~26.4(i)]{BC}.
\end{proof}

Next we define the set of least squares solutions
of $({D}_i)_{i\in I}$
\begin{equation}
\Lss=\argmin\
\sum_{i\in I}\dist{{D}_i}^2
.
\end{equation}

Finally, throughout the remainder of this section, we assume that
\begin{empheq}[box=\mybluebox]{equation}
\label{e:assump:bZnotfady}
\text{$\bzero\in \bDelta^\perp+\dom \bg^*$ and $\bZ\neq \varnothing $.}
\end{empheq}

\begin{remark}
	In  many applications,
	the individual  functions $g_i$ have  minimizers.
	In such cases, $(\forall i\in I)$ $0\in \dom 
  \partial g_i^*
  \subseteq \dom g_i^*$, and therefore
	$\bzero\in \dom \bg^*\subseteq \bDelta^\perp+\dom \bg^* $. 	
\end{remark}

\begin{proposition}
\label{prop:prod:sp}
The following hold:
\begin{enumerate}
\item
\label{prop:prod:sp:v}
 $\bv=\Pj{\overline{\bDelta-\dom \bg}}(\bzero)
 =\Pj{\overline{\bDelta-{\bD} }}(\bzero)
 \in \bDelta^\perp$.
 \item
\label{prop:prod:sp:v.0}
$\fix \Pj{\bDelta}\Pj{{\bD}}
=\bDelta\cap (\bv+{\bD})\neq\varnothing$.
 \item
\label{prop:prod:sp:v.5}
$(\forall y\in \fix \Pj{\bDelta}\Pj{{\bD}})$
$\bv=\by-\Pj{{\bD}}(\by)$.
 \item
 \label{prop:prod:sp:vi}
 $\bZ=\menge{\bx\in \bDelta}{\bDelta^\perp\cap\partial \bg (\bx-\bv)\neq\varnothing}
 =\bj\big(\zer \sum_{i\in I}\partial g_i(\cdot-v_i)\big )$.
  \item
 \label{prop:prod:sp:vi.0}
 $ \zer \Big(\sum_{i\in I}\partial g_i(\cdot-v_i)\Big)\neq \varnothing$.
  \item
 \label{prop:prod:sp:vi.5}
 $\Lss
=\fix \Big(\tfrac{1}{m}\sum_{i\in I}\Pj{{D}_i}\Big)
=\bigcap_{i\in I}(v_i+{D}_i).$
 \item
 \label{prop:prod:sp:vii}
 $\be(\bZ)=\zer \big(\sum_{i\in I}\partial g_i(\cdot-v_i)\big )
 \subseteq \cap_{i\in I} (\dom \partial g_i(\cdot-v_i))\subseteq \cap_{i\in I} 
 (v_i+{D}_i)=\Lss$.
\end{enumerate}
\end{proposition}
\begin{proof}
\ref{prop:prod:sp:v}:
Observe that
that $\overline{\bDelta-\dom \bg}
=\overline{\bDelta-\overline{\dom} \bg}
=\overline{\bDelta-\bD}$.
Now combine this with \cref{e:assump:bZnotfady}
and 
\cref{p:vUperp!}\ref{p:vUperp!:ii} applied with $(X,U,g)$
replaced by $(\bX,\bDelta,\bg)$.
\ref{prop:prod:sp:v.0}\&\ref{prop:prod:sp:v.5}:
Combine \cite[Lemma~2.2(i)\&(iv)]{BB94}
and \cref{e:0601c}
applied with $(X,U,g)$
replaced by $(\bX,\bDelta,\bg)$.
\ref{prop:prod:sp:vi}:
The first identity follows from 
applying \cref{e:Zvchar} with $(X,U,g)$
replaced by $(\bX,\bDelta,\bg)$.
The second identity follows from \cite[Proposition~26.4(vii)\&(viii)]{BC}.
\ref{prop:prod:sp:vi.0}:
This is a direct consequence of item \ref{prop:prod:sp:vi}.
\ref{prop:prod:sp:vi.5}:
Combine item \ref{prop:prod:sp:v},
\cite[Lemma~2.2(i)]{BB94}
and \cite[Corollary~3.1]{BDM:ORL16}.
\ref{prop:prod:sp:vii}:
This is a direct consequence of 
\ref{prop:prod:sp:vi} and \ref{prop:prod:sp:vi.5}.
\end{proof}

\begin{proposition}
	\label{prop:vi:zero}
	Suppose that $j\in I$ satisfies that $\dom g_j=X$.
	Then $v_j=0$.	
\end{proposition}
\begin{proof}
Set ${\bf A}=\argmin(\iota_{\bDelta}+\bg(\cdot-\bv))$
and observe that 
\cref{prop:prod:sp}\ref{prop:prod:sp:v}\&\ref{prop:prod:sp:v.0}
imply that
${\bf A}\subseteq \bDelta\cap (\bv+\dom \bg)
\subseteq \bDelta\cap (\bv+{\bD})
=\fix \Pj{\bDelta}\Pj{{\bD}}$.
Note that 
	 \cref{e:assump:bZnotfady} and \cref{0601c} 
   (applied with $(U,g)$ 
	 replaced by $(\bDelta, \bg)$)
	 imply that 
	${\bf A}=\bZ$.
  Hence, $\be({\bf A})=\be(\bZ)\subseteq \Lss$, 
  by \cref{prop:prod:sp}\ref{prop:prod:sp:vii}.
  Now, let $\by \in \fix \Pj{\bDelta}\Pj{{\bD}} $.
  Then  \cref{prop:prod:sp}\ref{prop:prod:sp:v.5} implies that 
  $\bv=\by-\Pj{{\bD}}(\by)
  =(y_1,\ldots,y_m)
  -(\Pj{{D}_1}y_1,\ldots,\Pj{{D}_m}y_m)$.     
Consequently, if 
$D_j=X$ then $v_j=y_j-\Pj{{D}_j}y_j=0$.
\end{proof}

\begin{theorem}
\label{thm:prod:sp}
Let $\bx=(x_i)_{i\in I}\in \bX$
and set $\by=\lim_{n\to\infty}\Pj{\fix \bT}(n\bv+\bT^n \bx)$.
Then 
\begin{equation}
\label{e:pd:main:i}
\Pj{\bDelta}\bT^n\bx \to \Pj{\bDelta}\by \in \argmin(\iota_\bDelta+\bg(\cdot-\bv)), 
\end{equation}
\begin{equation}
\label{e:pd:main:ii}
\text{$\bT^{n+1}\bx-\bT^n\bx+\Pj{\bDelta}\bT^n\bx=\Px{\bg}(\R{\bDelta}\bT^n\bx) 
\to -\bv+\Pj{\bDelta}\by$, and 
$\bg(\Px{\bg}\R{\bDelta}\bT^n\bx)\to\min(\iota_\bDelta+\bg(\cdot-\bv))$}.
\end{equation}
Furthermore,
\begin{equation}
\label{e:pd:main:iii}
\be(\Pj{\bDelta}\by)\in \argmin\Big(\sum_{i\in I} g_i(\cdot-v_i)\Big) .
\end{equation}
\end{theorem}
\begin{proof}
\cref{e:pd:main:i} and \cref{e:pd:main:ii} follow 
from applying \cref{main} with $(X,U,g)$
replaced by $(\bX,\bDelta,\bg)$. 
It follows from combining \cref{e:pd:main:i} and 
\cref{0601c} (applied with $(U,g)$ replaced by $(\bDelta,\bg)$)
that 
$\Pj{\bDelta}\by \in \argmin(\iota_\bDelta+\bg(\cdot-\bv))= \bZ$.
Now combine with 
\cref{prop:prod:sp}\ref{prop:prod:sp:vii}.
\end{proof}

\begin{corollary}
\label{cor:1}
Let 
$x_0\in X$,  and set 
$\overline{x}_0=x_{0,1}=\cdots=x_{0,m}=x_0$.
Update via $(\forall \nnn)$
\begin{subequations}
\label{e:update:ind}
\begin{align}
&(\forall i\in I)\quad x_{n+1,i}
=x_{n,i}-\overline{x}_n+\Px{g_i}(2\overline{x}_n-x_{n,i}),
\\
&\overline{x}_{n+1}=\tfrac{1}{m}\sum_{i\in I}x_{n+1,i}.
\end{align}
\end{subequations} 
Then $\overline{x}_n\to \overline{x}
\in \argmin\big(\sum_{i\in I} g_i(\cdot-v_i)\big)$. 
\end{corollary}
\begin{proof}
Combine \cref{thm:prod:sp} and 
\cref{prop:prod:sp}\ref{prop:prod:sp:vi.0}\&%
\ref{prop:prod:sp:ii.5}\&%
\ref{prop:prod:sp:iii}
in view of \cref{e:assump:bZnotfady}.
\end{proof}

\begin{corollary}
Suppose that $J\subseteq I$,
that
for every  $i\in \Iobj$, $f_i\colon X\to \RR$
is convex and satisfies $\dom f_i=X$ and $\argmin f_i\neq \fady$,
and that 
for every  $i\in \Icons$, $C_i\neq X$ is convex, closed, and nonempty.
Set $\Lss_C=\argmin\sum_{i\in \Icons}\dist{C_i}^2$.
Consider the problem
\begin{equation}
\label{e:cor:app}
\text{minimize $\sum_{i\in \Iobj} f_i(x)$ subject to
$x\in \bigcap _{i\in \Icons}C_i$.}
\end{equation}
Suppose  that $\zer\big(\sum_{i\in \Iobj}\partial f_i
+\sum_{i\in \Icons}\Nc{C_i}(\cdot-v_i)\big)\neq \varnothing$.
Let 
$x_0\in X$, and set 
$\overline{x}_0=x_{0,1}=\cdots=x_{0,m}=x_0$.
Update via $(\forall \nnn)$
\begin{subequations}
\label{e:update:split}
\begin{align}
&(\forall i\in \Iobj)\quad x_{n+1,i}=x_{n,i}-\overline{x}_n+\Px{g_i}(2\overline{x}_n-x_{n,i}),
\\
&(\forall i\in \Icons)\quad x_{n+1,i}=x_{n,i}-\overline{x}_n+\Pj{C _i}
(2\overline{x}_n-x_{n,i}),
\\
&\overline{x}_{n+1}=\tfrac{1}{m}\sum_{i\in I}x_{n+1,i}.
\end{align}
\end{subequations} 
Then $\overline{x}_n\to \overline{x}\in X$,
 and $\overline{x}$ is a solution of 
 \begin{equation}
 \label{e:p:lss:cons}
\text{minimize $\sum_{i\in \Iobj} f_i(x)$ subject to
$x\in \Lss_C$.}
\end{equation}
In particular, if $\cap _{i\in \Icons}C_i \neq \varnothing$,
then $\Lss_C=\cap _{i\in \Icons}C_i \neq \varnothing$
 and $\overline{x}$ is a solution of \cref{e:cor:app}.
\end{corollary}
\begin{proof}
Suppose that $g_i=f_i$, if $i\in \Iobj$;
and $g_i=\iota_{C_i}$, if $i\in \Icons$,
and observe that \cref{e:cor:app} reduces to
\begin{equation}
\text{minimize $\sum_{i\in I} g_i(x)$}.
\end{equation}
Note that combining 
\cref{e:update:ind} 
 and \cite[Example~23.4]{BC}
 yields \cref{e:update:split}.
It follows from \cref{prop:vi:zero} 
that $(\forall i\in \Iobj)$ $v_i=0$.
Consequently,
$\zer \big(\sum_{i\in I}\partial g_i(\cdot-v_i)\big) 
=\zer \big(\sum_{i\in \Iobj}\partial f_i
+ \sum_{i\in \Icons}\Nc{C_i}(\cdot-v_i)\big)\neq \varnothing$,
 and by \cref{cor:1} we have $\overline{x}_n\to \overline{x}\in X$,
 and $\overline{x}\in \zer \big(\sum_{i\in \Iobj}\partial f_i
+ \sum_{i\in \Icons}\Nc{C_i}(\cdot-v_i)\big)$.
Finally, using \cref{prop:prod:sp}\ref{prop:prod:sp:vi.5},
$(\exists u\in X)$
$-u\in \sum_{i\in \Iobj}\partial f_i(\overline{x})
=\partial (\sum_{i\in \Iobj}f_i)(\overline{x})$
 and $u\in \sum_{i\in \Icons}\Nc{C_i}(\overline{x}-v_i)
 \subseteq  \Nc{\cap_{i\in\Icons}(v_i+C_i)}(\overline{x})
 =\Nc{\Lss_C}(\overline{x})$.
Therefore, $\overline{x}$ solves \cref{e:p:lss:cons}. 
\end{proof}

\small 

\section*{Acknowledgements}
The authors thank the editor and three anonymous referees for 
insightful comments that led to a substantially improved manuscript. 
The research of HHB was partially supported by a Discovery Grant
of the Natural Sciences and Engineering Research Council of
Canada. 
The research of WMM was partially supported by 
the Natural Sciences and Engineering Research Council of
Canada Postdoctoral Fellowship.


\begin{thebibliography}{999}


\seppthree



\bibitem{Banjac}
G.\ Banjac, P.\ Goulart, B.\ Stellato, and S.\ Boyd,
Infeasibility detection in the alternating direction 
method of multipliers for convex optimization,
\emph{Journal of Optimization Theory and Applications}~183 (2019),
490--519. 


\bibitem{BB94}
H.H.\ Bauschke and J.M.\ Borwein,
Dykstra's alternating projection algorithm 
for two sets, 
\emph{Journal of Approximation 
Theory}~79 (1994), 418--443.

\bibitem{BBL97}
H.H.\ Bauschke, J.M.\ Borwein, and A.S.\ Lewis, 
The method of cyclic projections for closed
convex sets in Hilbert space, in \emph{Recent Developments in Optimization Theory and Nonlinear
Analysis (Jerusalem 1995), Contemporary Mathematics}~204 (1997), 1--38.

\bibitem{74}
H.H.\ Bauschke, R.I.\ Bo\c{t}, W.L.\ Hare, and W.M.\ Moursi,
Attouch-Th\'era duality revisited: 
paramonotonicity and operator splitting,
\emph{Journal of Approximation Theory}~164 (2012), 1065--1084.




\bibitem{BC}
H.H.\ Bauschke and P.L.\ Combettes,
\emph{Convex Analysis and Monotone Operator Theory in Hilbert Spaces},
2nd edition, Springer, 2017.

\bibitem{Lukepaper}
H.H.\ Bauschke, P.L.\ Combettes, and D.R.\ Luke,
Finding best approximation pairs relative to two 
closed convex sets in Hilbert spaces,
\emph{Journal of Approximation Theory}~127 (2004), 178--192. 

\bibitem{BDM:ORL16}
H.H.\ Bauschke, M.N.\ Dao, and W.M.\ Moursi, 
The Douglas--Rachford algorithm in the affine-convex case, 
\emph{Operations Research Letters}~44  (2016) 
379--382.


\bibitem{Sicon}
H.H.\ Bauschke, W.L.\ Hare, and W.M.\ Moursi,
Generalized solutions for the sum of two maximally
monotone operators,
\emph{SIAM Journal on Control and Optimization}~52 (2014), 1034--1047. 

\bibitem{MOR}
H.H.\ Bauschke, W.L.\ Hare, and W.M.\ Moursi,
On the range of the Douglas--Rachford operator,
\emph{Mathematics of Operations Research}~41 (2016), 884--897. 

\bibitem{BMWnear}
H.H.\ Bauschke, S.M.\ Moffat, and X.\ Wang,
Near equality, near convexity, sums of maximally monotone operators, 
and averages of firmly nonexpansive mappings, 
\emph{Mathematical Programming (Series~B)}~139 (2013),
55--70.

\bibitem{101}
H.H.\ Bauschke and W.M.\ Moursi,
The Douglas--Rachford algorithm for two 
(not necessarily intersecting) affine subspaces,
\emph{SIAM Journal on Optimization}~26 (2016), 968--985.


\bibitem{BM:MPA17}
H.H.\ Bauschke and W.M.\ Moursi, 
On the Douglas--Rachford algorithm, 
\emph{Mathematical Programming (Series A)}~164 (2017),  
263--284.

\bibitem{BM:order15}
H.H.\ Bauschke and W.M.\ Moursi, 
On the order of the operators in the Douglas--Rachford algorithm, 
\emph{Optimization Letters}~10 (2016),  
447--455.

\bibitem{BDM:16}
H.H.\ Bauschke, M.M.\ Dao and W.M.\ Moursi, 
The Douglas--Rachford algorithm in the affine-convex case, 
\emph{Operations research Letters}~44 (2016),  
379--382.

\bibitem{Comb09}
P.L.\ Combettes,
Iterative construction of the resolvent of a sum of maximal monotone operators,
\emph{Journal of Convex Analysis}~16 (2009), 727--748. 

\bibitem{DR}
J.\ Douglas and H.H.\ Rachford,
On the numerical soluion of heat conduction problems in two and
three variables, 
\emph{Transactions of the AMS}~82 (1956), 421--439. 

\bibitem{EckBer}
J.\ Eckstein and D.P.\ Bertsekas,
On the Douglas-Rachford splitting method and the proximal
point algorithm for maximal monotone opeators,
\emph{Mathematical Programming (Series A)}~55 (1992), 293--318. 

\bibitem{Iusem98}
A.N.\ Iusem,
On some properties of paramonotone operators,
\emph{Journal of Convex Analysis}~5 (1998), 269--278.

\bibitem{KN}
W.J.\ Kaczor and M.T.\ Nowak,
\emph{Problems in Mathematical Analysis I},
AMS, Providence, Rhode Island, 2000. 

\bibitem{Knopp}
K.\ Knopp,
\emph{Infinite Sequences and Series},
Dover, New York, 1956. 

\bibitem{LM}
P.-L.\ Lions and B.\ Mercier,
Splitting algorithms for the sum of two
nonlinear operators,
\emph{SIAM Journal on Numerical Analysis}~16 (1979), 964--979. 

\bibitem{Liu}
Y.\ Liu, E.K.\ Ryu, and W.\ Yin, 
A new use of Douglas-Rachford splitting for identifying infeasible, 
unbounded, and pathological conic programs,
\emph{Mathematical Programming (Series~A)}~177 (2019), 225--253. 

\bibitem{MMW:2015}
S.M.\ Moffat, W.M.\ Moursi and S.\ Wang,
Nearly convex sets: fine properties and domains 
or ranges of subdifferentials of convex functions, 
\emph{Mathematical Programming (Series~A)}~126 (2016),
193--223.

\bibitem{Rock70}
R.T.\ Rockafellar,
\emph{Convex Analysis},
Princeton University Press, Princeton, 1970.


\bibitem{ucla}
E.K.\ Ryu, Y.\ Liu, and W.\ Yin,
Douglas-Rachford splitting and ADMM
for pathological convex optimization,
\emph{Computational Optimization and Applications}~74 (2019),
747--778, 
\url{https://doi.org/10.1007/s10589-019-00130-9}
and also \texttt{arxiv:1801.06618}

\bibitem{Svaiter}
B.F.\ Svaiter, 
On weak convergence of the Douglas-Rachford method,
\emph{SIAM Journal on Control and Optimization}~49 (2011),
280--287.




\end{thebibliography}
\end{document}